\documentclass[12pt]{amsart}
\usepackage[margin=.75in]{geometry}
\usepackage{amsmath}
\usepackage{amsxtra}
\usepackage{amscd}
\usepackage{amsthm}
\usepackage{amssymb}
\usepackage{youngtab}
\usepackage[normalem]{ulem}
\usepackage{comment}
\usepackage{color}
\usepackage[usenames,dvipsnames]{xcolor}
\usepackage{tikz}
\usetikzlibrary{matrix,arrows,decorations.pathmorphing}
\usetikzlibrary{shapes.multipart}
\usepackage{hyperref}
\usepackage{bbm}
\usepackage{ytableau}
\usepackage{blkarray}
\usepackage{anyfontsize}

\numberwithin{equation}{section}

\newtheorem{theorem}{Theorem}[section]
\newtheorem{lemma}[theorem]{Lemma}

\newtheorem{corollary}[theorem]{Corollary}
\newtheorem{conjecture}[theorem]{Conjecture}

\theoremstyle{definition}

\newtheorem{example}[theorem]{Example}

\newtheorem{question}[theorem]{Question}

\theoremstyle{remark}
\newtheorem{remark}[theorem]{Remark}

\newcommand{\CC}{\mathbf{C}}

\newcommand{\ZZ}{\mathbf{Z}}
\newcommand{\NN}{\mathbb{N}}

\newcommand{\OO}{\mathcal{O}}

\newcommand{\hh}{\mathfrak{h}}

\newcommand{\la}{\langle}
\newcommand{\ra}{\rangle}

\newcommand{\ttt}{\mathfrak{t}}

\newcommand{\Triv}{\mathrm{Triv}}

\newcommand{\ct}{\mathrm{ct}}
\newcommand{\hd}{\mathrm{hd}}
\newcommand{\emp}{\emptyset}
\newcommand{\lra}{\longrightarrow}

\newcommand{\Del}{\Delta}

\DeclareMathOperator{\Irr}{Irr}

\DeclareMathOperator{\Ext}{Ext}
\DeclareMathOperator{\Hom}{Hom}

\DeclareMathOperator{\Rad}{Rad}
\DeclareMathOperator{\KZ}{KZ}
\DeclareMathOperator{\rad}{rad}
\DeclareMathOperator{\im}{Im}

\DeclareMathOperator{\End}{End}

\begin{document}

\title[Character formulas and BGG resolutions]{Character formulas and Bernstein-Gelfand-Gelfand resolutions for Cherednik algebra modules}

\begin{abstract} We study blocks of category $\OO$ for the Cherednik algebra having the property that every irreducible module in the block admits a BGG resolution, and as a consequence prove a character formula conjectured by Oblomkov-Yun. \end{abstract}

\author{Stephen Griffeth}
\address{Stephen Griffeth \\
Instituto de Matem\'atica y F\'isica \\
Universidad de Talca}
\email{}

\author{Emily Norton}
\address{Emily Norton \\
Department of Mathematics \\
Kansas State University}

\maketitle

\section{Introduction} A fundamental problem in representation theory is the calculation of characters of irreducible objects. For category $\OO$ of the rational Cherednik algebra the complete solution is equivalent to calculating the decomposition matrix $[\Delta(\lambda):L(\mu)]$. For the cyclotomic rational Cherednik algebras, the entries of the decomposition matrix are values of affine parabolic Kazhdan-Lusztig polynomials \cite{RSVV}, \cite{Lo}, \cite{We}; this amounts, in principle, to a recursive algorithm for calculating characters. Special classes of modules exist, however, for which we can write closed character formulas. Our work in the present paper was motivated by the desire to understand the reasons for the existence of such formulas in a particular class of examples, those appearing in Conjecture 9.10.1 of Oblomkov-Yun \cite{ObYu}:
 $$\sum_{n>0} \mathrm{dim}_\CC(L_{1/2n}(D_{2n}))x^{n-1}=(1-4x)^{-3/2} $$ and $$\sum_{n>0} \mathrm{dim}_\CC(L_{1/2n}(C_{2n}))x^n=(1-4x)^{-3/2}(1+\sqrt{1-4x})^2/4.$$ In these formulas, $L_c(W)$ denotes the $H_c(W)$ module $L_c(\Triv)$. This conjecture is a consequence of the following graded dimension formula (see Corollary \ref{dim formula}):
\begin{equation} \label{graded dim}
\mathrm{dim}_\CC(L^d)=\dim_\CC(L^{N-d})=\sum_{\substack{0 \leq \ell \leq d \\ d-\ell \in 2 \ZZ}}  {2n \choose \ell}
\end{equation}
if $0 \leq d \leq n$, $N=2n$, and $W=W(C_{2n})$; or if $0 \leq d \leq n-1$, $N=2n-2$, and $W=W(D_{2n})$; and in which $L^d$ denotes the degree $d$ part of $L=L_{1/2n}(\Triv)$.
 
In most branches of Lie theory, those irreducible modules for which character formulas as simple as this exist tend to admit \emph{BGG resolutions}, that is, resolutions by standard (also known as Verma) modules. Not only do the modules involved in \eqref{graded dim} admit BGG resolutions, but so do \emph{all} the simples in the same block. 
Our first theorem characterizes this situation and gives a practical method for proving that BGG resolutions exist (see Theorem \ref{tight 2} for a more general result which would also apply to the module category of a cellular algebra):
\begin{theorem} \label{tight 1}
Let $C$ be a highest weight category over a field $F$, with simple objects indexed by a finite poset $\Lambda$. The following conditions are equivalent:
\begin{itemize}
\item[(a)] Every simple object in $C$ admits a resolution by standard objects.
\item[(b)] The radical of every standard object in $C$ is generated by homomorphic images of standard objects.
\item[(c)] $[\Delta(\lambda):L(\mu)]=\dim\Hom(\Delta(\mu),\Delta(\lambda))$ for all $\lambda,\mu\in\Lambda$. 
\end{itemize}
\end{theorem} 
Earlier known examples of blocks of category $\OO_c$ satisfying these conditions were closely related to Haiman's work on diagonal coinvariant rings; see \cite{BEG}, \cite{Gor}, \cite{GoGr}.

The proof we give of \eqref{graded dim} utilizes the tools developed in \cite{Gri}, which completely describe the submodule structure of the standard module $\Delta_c(\Triv)$ when the parameter $c_0$ associated with the conjugacy class of reflections containing the transpositions is not a positive rational number of denominator less than $n$. The most interesting case is when $c_0=\ell/n$ for a positive integer $\ell$ coprime to $n$. In this case, we can give a very complete description of the structure of the standard modules in the principal block of category $\OO_c$ for the rational Cherednik algebra of the group $G(r,1,n)$, and in particular we can compute the graded dimensions of their simple heads. Our next theorem specifies the set of parameters $c$ for which the conditions in Theorem \ref{tight 1} hold for the principal block of category $\OO_c$. We order pairs of integers lexicographically: $(i_1,k_1) > (i_2,k_2)$ if $k_1>k_2$ or $k_1=k_2$ and $i_1>i_2$. We mark the boxes of the trivial partition $(n)$ as follows: for a box $b$, if there exist integers $0 \leq i \leq r-1$ and $k$ so that $$(d_i-i)/r=d_0/r+\mathrm{ct}(b) c_0+k,$$ then we place a label $(i,k)$ in the box $b$ (so a box may receive more than one label, but a given $i$ appears at most once). An example of such a labeling when $n=8$ and $r=6$ is

\vspace{1em}
\ytableausetup{mathmode, boxsize=3.5em}
\begin{ytableau}
\begin{matrix} (0,0) \\ (1,-2) \end{matrix} & (3,-1) & & & & \begin{matrix} (5,3) \\ (2,-2) \end{matrix} &  & 
\end{ytableau}

Using this notation we give the following characterization of when the conditions from Theorem \ref{tight 1} hold for the principal block (see Theorem \ref{tight thm} for the proof):
\begin{theorem} \label{intro main}
Suppose $c_0=\ell/n$ for a positive integer $\ell$ coprime to $n$. Then every simple module in the principal block has a BGG resolution if and only if when read left to right, and assuming the labels appearing in a given box are ordered in decreasing fashion, we have
$$(i_1,k_1)>(i_2,k_2)>\cdots>(i_s,k_s)>(i_1,k_1-\ell).$$
\end{theorem} Thus in the example preceding the theorem there are simples in the principal block not admitting BGG resolutions. On the other hand, it is straightforward to check that for the parameters appearing in \eqref{graded dim}, the condition of Theorem \ref{intro main} is satisfied: for the type $D_{2n}$ version, take $c_0=1/2n$ and $d_0=d_1=0$, in which case the labels are $(0,0)$ in the first box, and $(1,-1)$ in the box with content $n$; for the type $C_{2n}$ version, take $c_0=d_0=-d_1=1/2n$ so that the labels are $(0,0)$ appearing in the first box and $(1,-1)$ appearing in the box with content $n-1$.

When the principal block $B_0$ of $\OO_c$ satisfies the conditions in Theorems \ref{tight 1} and \ref{intro main}, we can conjecturally calculate BGG resolutions of simples using a graph $\Gamma$ depicting the submodule structure of standard modules. The graph $\Gamma$ has a vertex for each $\lambda\in B_0$, and an arrow $\mu\rightarrow \lambda$ for each $\mu$ which is the highest weight of a submodule of $\Del(\lambda)$ of the type defined in Theorem \ref{submodule structure}. These generate the radical of $\Del(\lambda)$ (but not always minimally). We conjecture (\ref{bgg conj}) that the standards in the BGG resolution of $L(\lambda)$ may be found by, first, taking the subgraph $\Gamma_\lambda$ spanned by all $\mu$ such that there is a path from $\mu$ to $\lambda$ in $\Gamma$; and then, deleting all subgraphs $\Gamma_\nu$ whenever a single arrow $\nu\rightarrow\mu$ can be factored as $\nu\rightarrow\rho\rightarrow\mu$ for vertices $\nu,\mu,\rho\in\Gamma_\lambda$. We expect that the vertices $\mu$ which remain after this procedure are those such that $\Del(\mu)$ appears in the minimal BGG resolution of $L(\lambda)$; moreover, $\Delta(\mu)$ occurs with multiplicity $1$ in homological degree $d(\mu,\lambda)$, where $d(\mu,\lambda)$ is the length of the longest chain of directed arrows from $\mu$ to $\lambda$ in $\Gamma$. For $r=2$ and equal parameters $1/2n$, we can write down all maps between standard modules and we verify (\ref{bgg B2n}) that this algorithm is indeed correct. We also have calculated some examples for $G(r,1,rn)$ at equal parameters $1/rn$; the conjectural algorithm produces the correct graded characters for simple modules in the blocks in our examples, see Section \ref{examples}.

For the Cherednik algebra, BGG resolutions that are minimal in a certain sense are minimal resolutions in the sense of commutative algebra; \cite{BGS} used this observation to conjecture a combinatorial formula for the Betti numbers of the ideals of certain subspace arrangements. In this direction, here we develop new tools for establishing the existence of BGG resolutions, thus reducing the calculation of Betti numbers of a number of subspace arrangements to the calculation of character formulas for Cherednik algebra modules. We expect that in our categories of interest satisfying the conditions of Theorem \ref{tight 1}, the character of a simple module should determine the terms in its minimal BGG resolution. For example, we have:
\begin{theorem}\label{intro bgg thm}
Suppose $W=B_{2n}=G(2,1,2n)$ and the parameters $c=1/2n$ are constant.

The module $L_c(\mathrm{triv})=L((2n),\emp)$ is the coordinate ring of the scheme-theoretic intersection of $$\{x_1^2=x_2^2=\cdots=x_{2n}^2 \}$$ and the set of points where at least $n$ coordinates are equal to $0$. For $0\leq i \leq n$, the coefficient of  $v^i$ in the graded character of $L_c(\Triv)$ is the sum $(-1)^i\Del_i$ where 
$$\Del_i=\Del((2n-i,1^i),\emp)\oplus\bigoplus_{1 \leq a \leq i}\Del((n-a,1^{i-a}),(n+a+1-i,1^{a-1}))$$
and the $(2n-i)$'th term is the sum $(-1)^{2n-i}\Del_{2n-i}$ where 
$$\Del_{2n-i}=\Del(\emp,(2n-i,1^i)^t)\oplus\bigoplus_{1 \leq a \leq i}\Del((n+a+1-i,1^{a-1})^t,(n-a,1^{i-a})^t)$$

\end{theorem}

\begin{conjecture}\label{intro bgg conj}
For $0\leq i \leq n$, the $i$'th term of the minimal BGG resolution of $L_c((2n),\emp)$ is $\Del_i$ and the $(2n-i)$'th term is $\Del_{2n-i}$. 
\end{conjecture}

\textbf{Acknowledgments.} We acknowledge Fondecyt proyecto 1110072 for financing E.N.'s visit to Chile in March 2015 where this project started, and S.G. acknowledges Fondecyt proyecto 1151275 for continuing financial support. E.N. is grateful to O. Dudas and H. Miyachi for enlightening discussions. We thank S. Arkhipov, R. Bezrukavnikov, V. Ginzburg, and I. Loseu for helpful remarks about graded BGG reciprocity, and C. Berkesch-Zamaere and S. Sam for sharing their expertise in commutative algebra.

\section{BGG resolutions}

\subsection{} In this section we give definitions, fix notation, and prove some very general results on BGG resolutions in a class of categories containing highest weight categories with finitely many isoclasses of simple objects.

\subsection{Axiomatics} \label{axioms} Let $F$ be a field and let $C$ be an artinian $F$-linear abelian category. Assume $\Lambda$ is a finite poset such that for each $\lambda \in \Lambda$, there is an object $\Delta(\lambda)\in C$ satisfying:
\begin{itemize}
\item[(a)] The objects $L(\lambda)=\Del(\lambda)/\Rad(\Delta(\lambda))$ are a complete set of non-isomorphic simple objects of $C$,
\item[(b)] we have $\End(\Del(\lambda))=F=\End(L(\lambda))$, and
\item[(c)] if $[\Rad(\Del(\lambda)):L(\mu)] \neq 0$ then $\mu < \lambda$.
\end{itemize} For instance, if $C$ is a highest weight category in the sense of \cite{CPS} with finitely many isoclasses of simple objects, then it satisfies these axioms. We will refer to the objects $\Delta(\lambda)$ as \emph{standard objects}.

\subsection{Tight multiplicities and singular vectors} Let $C$ be a category as in \ref{axioms}.
By right exactness of Hom we have
\begin{equation} \label{hom inequality}
\dim\Hom(\Del(\lambda),M) \leq \dim\Hom(\Del(\lambda),N)+\dim\Hom(\Del(\lambda),M/N)
\end{equation} for each object $M$ in $C$ and every subobject $N$ of $M$. Thus by induction on the length of $M$ we have 
$$\dim\Hom(\Del(\lambda),M)\leq [M:L(\lambda)].$$ An object $M$ in $C$ has \emph{tight multiplicities} if for all $\lambda \in \Lambda$ we have $$\dim\Hom(\Del(\lambda),M)=[M:L(\lambda)].$$ Thus if $M$ has tight multiplicities, all occurrences of an irreducible factor $L(\lambda)$ in a composition series for $M$ are accounted for by homomorphisms from $\Del(\lambda)$ to $M$. 

\begin{lemma} \label{inherit tight}
If $M$ has tight multiplicities, so does every subobject and every quotient object. If $M$ and $N$ have tight multiplicities, so does $M \oplus N$.
\end{lemma}
\begin{proof}
It is obvious that direct sums of objects with tight multiplicities have tight multiplicities. The proofs for quotient objects and subobjects are dual to one another; we prove that every subobject inherits tight multiplicities.

Suppose $M$ has tight multiplicities and that $N \subseteq M$ is a subobject. We prove that $N$ has tight multiplicities by induction on the length of $M/N$, which we may therefore assume to be $1$. If $M/N \cong L(\lambda)$, then using \eqref{hom inequality} we obtain
$$\dim\Hom(\Del_c(\mu),N) \geq \dim\Hom(\Del_c(\lambda),M)-\delta_{\lambda \mu}$$ where $\delta_{\lambda \mu}$ is $0$ unless $\lambda=\mu$ in which case it is $1$. This proves the result.  \end{proof}

An object $M$ in $C$ is \emph{generated by singular vectors} if it is equal to the largest subobject $N$ of $M$ containing the images of all homomorphisms from standard objects $\Delta(\lambda)$ to $M$ (one could also write \emph{generated by highest weight vectors}, the more common terminology in classical Lie theory).

\begin{lemma} \label{tight iff sing}
Suppose $M$ is an object in $C$. Then $M$ has tight multiplicities if and only if every submodule of $M$ is generated by singular vectors.
\end{lemma}
\begin{proof}
First suppose $M$ has tight multiplicities. Then by Lemma \ref{inherit tight} the same is true of every subobject, so it suffices to prove that if $M$ has tight multiplicities then it is generated by singular vectors. Let $N$ be the subobject of $M$ generated by the images of homomorphisms from standard objects to $M$. For each $\lambda \in \Irr W$ we have
\begin{align*}
[M:L(\lambda)]&=\dim\Hom(\Del(\lambda),M)=\dim\Hom(\Del(\lambda),N) \\ &\leq [N:L(\lambda)] \leq [M:L(\lambda)], \end{align*} implying equality throughout and hence $N=M$.

Conversely, assume that every subobject of $M$ is generated by singular vectors. We prove by induction on the length of $M$ that $M$ has tight multiplicities. This is clear if the length of $M$ is $1$. Let $$0=M_0 \subseteq M_1 \subseteq \cdots \subseteq M_\ell=M$$ be a composition series. By assumption every subobject of $M_{\ell-1}$ is generated by singular vectors; by induction, $M_{\ell-1}$ has tight multiplicities. This implies that
$$\sum_\lambda \dim\Hom(\Del(\lambda),M)) \geq \ell-1.$$ But if equality holds then the subobject $N$ of $M$ generated by homomorphisms from standard objects is contained in $M_{\ell-1}$, contradicting that $M$ is generated by singular vectors.
\end{proof}

\begin{lemma} \label{rad enough}
If the radical of each $\Delta(\lambda)$ is generated by singular vectors, then every subobject of every standard object is generated by singular vectors.
\end{lemma}
\begin{proof}
We proceed by induction on the poset $\Lambda$. If $\lambda$ is minimal then $L(\lambda)=\Del(\lambda)$ is simple and hence every subobject is generated by singular vectors. In general, we have $$L(\lambda)=\Del(\lambda)/\Rad\Del(\lambda)$$
with $\Rad \Del(\lambda)$ generated by singular vectors by hypothesis, so there is a sum $\Del_1$ of standard objects surjecting onto $\Rad \Del(\lambda)$. Using the axioms \ref{axioms} we may assume that the standard summands in $\Del_1$ are all smaller than $\lambda$. So the inductive hypothesis and Lemmas \ref{inherit tight} and \ref{tight iff sing} together imply that every subobject of $\Del_1$ is generated by singular vectors, and hence the same holds for its quotient $\Rad\Del(\lambda)$. But every proper subobject of $\Del(\lambda)$ is a subobject of $\Rad\Del(\lambda)$, so every subobject of $\Del(\lambda)$ is generated by singular vectors.
\end{proof}

\subsection{BGG resolutions}

A \emph{Bernstein-Gelfand-Gelfand resolution} (or briefly, \emph{BGG resolution}) of an object $M$ of $C$ is  an exact sequence  $$\cdots \longrightarrow \Del_m \lra  \Del_{m-1} \lra \cdots \lra \Del_1 \lra \Del_0\lra M \lra 0$$ in which each $\Del_i$ is a (possibly empty) finite direct sum of standard objects. It is \emph{minimal} if the differential maps $\Delta_i$ into $\mathrm{Rad}(\Delta_{i-1})$. We observe that a minimal BGG resolution necessarily has $\Del_i=0$ for $i$ sufficiently large: since the radical of $\Del_i$ is the sum of the radicals of its standard summands, our axioms \ref{axioms} imply that for each standard summand $\Del(\lambda)$ of $\Del_i$ there is a standard summand $\Del(\mu)$ of $\Del_{i-1}$ with $\lambda < \mu$.

\begin{lemma} \label{minimal lemma}
If $\Del_1 \rightarrow \Delta_0 \rightarrow M \rightarrow 0$ is exact, where $\Del_1=\bigoplus_{i \in I} \Delta(\lambda_i)$ and $\Del_0=\bigoplus_{j \in J} \Delta(\mu_j)$ are finite sums of standard objects, then there are subsets $I_0 \subseteq I$ and $J_0 \subseteq J$ and an exact sequence 
$$\bigoplus_{i \in I_0} \Delta(\lambda_i) \rightarrow \bigoplus_{j \in J_0} \Delta(\mu_j) \rightarrow M \rightarrow 0$$ so that the image of $\bigoplus_{i \in I_0} \Delta(\lambda_i)$ is contained in $\bigoplus_{j \in J_0} \mathrm{Rad}(\Delta(\mu_j))$.
\end{lemma}
\begin{proof}
If the image of $\Delta(\lambda_{i_0})$ is not contained in $\bigoplus_{j \in J} \mathrm{Rad}(\Delta(\mu_j))$ then there is some index $j_0 \in J$ so that the composite $\Delta(\lambda_{i_0}) \rightarrow \Delta(\mu_{j_0}) \rightarrow L(\mu_{j_0})$ is non-zero. Since the top of $\Delta(\lambda_{i_0})$ is $L(\lambda_{i_0})$, it follows that $\lambda_{i_0}=\mu_{j_0}$ and hence the map $\Delta(\lambda_{i_0}) \rightarrow \Delta(\mu_{j_0})=\Delta(\lambda_{i_0})$ is a scalar multiple of the identity. In particular the map $\Delta(\lambda_{i_0}) \rightarrow \bigoplus \Delta(\mu_j)$ is injective. Hence the sequence
$$\bigoplus_{i \neq i_0} \Delta(\lambda_i) \rightarrow \left(\bigoplus_{j \in J} \Delta(\mu_j) \right)/\mathrm{Im}(\Delta(\lambda_{i_0})) \rightarrow M \rightarrow 0$$ is exact. Since $$\mathrm{Im}(\Delta(\lambda_{i_0})) \cap \bigoplus_{j \neq j_0} \Delta(\mu_j)=0$$ the map
$$\bigoplus_{j \neq j_0} \Delta(\mu_j) \rightarrow \left(\bigoplus_{j \in J} \Delta(\mu_j) \right)/\mathrm{Im}(\Delta(\lambda_{i_0}))$$ is injective; it is therefore an isomorphism because both sides have the same length. We may thus eliminate $i_0$ and $j_0$ from our index sets. Repeating this procedure proves the lemma.
\end{proof}

If a BGG resolution exists, then $M$ is generated by singular vectors. If an irreducible object $L(\lambda)$ has a BGG resolution, then the first map has domain $\Del_0$ a sum of standard objects; since the only standard object with a non-zero map to $L(\lambda)$ is $\Del(\lambda)$, and since there is a unique such map up to scalars, we may discard all the summands of $\Del_0$ and assume the first step in the resolution is the quotient map $\Del(\lambda) \rightarrow L(\lambda)$. Therefore if $L(\lambda)$ has a BGG resolution then $\Rad \Del(\lambda)$ is generated by singular vectors.

\begin{lemma} \label{sing iff BGG}
If every submodule of every standard module in $C$ is generated by singular vectors then every quotient of a sum of standard modules has a minimal BGG resolution. 
\end{lemma}
\begin{proof}
First suppose that every subobject of every standard object in $C$ is generated by singular vectors. By Lemmas \ref{inherit tight} and \ref{tight iff sing}, every subobject of every sum of standard objects in $C$ is generated by singular vectors. Let $\Del_0$ be a sum of standard modules and let $M=\Del_0/N$ be a quotient of $\Del_0$. By hypothesis $N$ is generated by singular vectors, so it is a quotient of a finite sum $\Del_1$ of standard objects. We have obtained an exact sequence
$$ \Delta_1 \rightarrow \Del_0 \rightarrow M \rightarrow 0,$$ and by lemma \ref{minimal lemma} we may assume that the image of $\Del_1$ is contained in the radical of $\Del_0$. Iterating produces a minimal BGG resolution. 
\end{proof}

The category $C$ has \emph{tight multiplicities} if every standard object in $C$ does. Examples include defect one blocks in Cherednik algebra category $\OO$ (see 5.2.4 of \cite{Rouq})). 
We now restate and prove Theorem \ref{tight 1} using these definitions.

\begin{theorem} \label{tight 2} Assume $C$ is a category as in \ref{axioms}. The following are equivalent:
\begin{enumerate}
\item $C$ has tight multiplicities.
\item Every simple object in $C$ has a BGG resolution.
\item The radical of every standard object in $C$ is generated by singular vectors.
\end{enumerate} 
\end{theorem}
\begin{proof} (a) implies (b) follows from Lemma \ref{tight iff sing} and Lemma \ref{sing iff BGG}. (b) implies (c) is the discussion following the definition of BGG resolutions. (c) implies (a) follows from Lemma \ref{rad enough} and Lemma \ref{tight iff sing}. \end{proof} 

For finite dimensional complex semisimple Lie algebras, there is another way of stating this criterion.
\begin{corollary}
Let $C$ be a block of category $\OO$ for a finite dimensional semisimple Lie algebra over $\CC$. Then every simple in $C$ has a BGG resolution if and only if every Verma module is multiplicity free.
\end{corollary}
\begin{proof}
It is a consequence of classical Lie theory that $$[\Delta(\lambda):L(\mu)] \neq 0 \ \iff \mathrm{Hom}(\Delta(\mu),\Delta(\lambda)) \neq 0$$ and that in this case $\dim\Hom(\Delta(\mu),\Delta(\lambda))=1$. It follows that $C$ has tight multiplicities if and only if every Verma module is multiplicity free.
\end{proof} We remark that K\"{o}nig \cite{Koe} proves this corollary by a different method; we do not know if the standing hypotheses of his paper hold in the cases of interest to us.

\subsection{BGG resolutions and character formulas}

Suppose $M$ is a module in $C$ that admits a BGG resolution
$$0 \rightarrow \Delta_\ell \rightarrow \cdots \rightarrow \Delta_0 \rightarrow M$$ 

In the Grothendieck group of $C$ this gives
$$[M]=\sum (-1)^i [\Delta_i]=\sum (-1)^i c_{\lambda,i} [\Delta_c(\lambda)],$$ where
$$\Delta_i \cong \bigoplus_{\lambda} \Delta_c(\lambda)^{\oplus c_{\lambda,i}}.$$ 

\subsection{Wreath products} In \cite{ChTa}, Chuang and Tan study filtrations of modules of wreath product algebras and the corresponding graded decomposition numbers. Here we combine their results with ours to show that the conditions of Theorem \ref{tight 2} are inherited by wreath products as well.

Let $A$ be a finite-dimensional $F$-algebra, let $I$ be a finite partially ordered set, and for each $i \in I$ let $\Delta(i)$ be an $A$-module. With $C=A \mathrm{-mod}$ we assume that the conditions from \ref{axioms} hold. Let $n$ be a positive integer and write $$A_n=A^{\otimes n} \rtimes S_n$$ for the wreath product of $A$ with $S_n$. A collection $\lambda=(\lambda^i)_{i \in I}$ of partitions indexed by $i \in I$ with $\sum_{i,j} \lambda^i_j=n$ is an \emph{$I$-partition of $n$}. Given an $I$-partition $\lambda$ of $n$ and a collection of $A$-modules $M(i)$, $i \in I$, we write $A_\lambda=\bigotimes_{i \in I} A_{|\lambda^i|}$ and define
$$M(\lambda)=\mathrm{Ind}_{A_\lambda}^{A_n} \bigotimes_{i \in I} M(i)^{\otimes |\lambda^i|} \otimes S^{\lambda_i}.$$ Given two $I$-partitions $\lambda$ and $\mu$ of $n$, we write $\lambda \leq \mu$ if $\lambda=\mu$ or for each $i \in I$,
$$\sum_{j \geq i} |\lambda^j| \leq \sum_{j \geq i} |\mu^j|,$$ with a strict inequality for some $i \in I$. This defines a partial order on the set of $I$-partitions of $n$.
\begin{lemma}
If $A$-mod, with standard objects $\Delta(i)$, satisfies the conditions from \ref{axioms}, then the category $A_n$-mod satisfies the conditions from \ref{axioms}, with standard objects $\Delta(\lambda)$ indexed by the poset of $I$-partitions $\lambda$ of $n$.
\end{lemma}
\begin{proof}
Proposition 3.7 of \cite{ChTa} implies that axioms (a) and (b) of \ref{axioms} hold for $A_n$, while Proposition 4.7 of \cite{ChTa} implies that (c) holds. 
\end{proof}

\begin{lemma} \label{standard products}
Let $i$ and $j$ be positive integers with $i+j=n$, let $\mu$ be an $I$-partition of $i$ and let $\lambda$ be an $I$-partition of $j$. Then the module $\mathrm{Ind}_{A_i \otimes A_j}^{A_n} \Delta(\mu) \otimes \Delta(\lambda)$ is a direct sum of standard modules for $A_n$.
\end{lemma}
\begin{proof}
This follows from the definition of standard module together with part (1) of Lemma 3.3 of \cite{ChTa}.
\end{proof}

\begin{corollary}\label{wreath}
If $A$ is a finite dimensional $F$-algebra, $I$ is a finite poset, and $\Delta(i)$, for $i \in I$, are $A$-modules so that the conditions of \ref{axioms} and the conclusions of Theorem \ref{tight 2} hold, then they hold for $A_n$ with standard modules $\Delta(\lambda)$, where $\lambda$ ranges over all $I$-partitions of $n$.
\end{corollary}
\begin{proof}
It suffices to prove that the radical of each $\Delta(\lambda)$ is a homomorphic image of a sum of $\Delta(\mu)$'s. We first treat the special case $\lambda^i=\emptyset$ for all but one $i \in I$. In this case $\Delta(\lambda)=\Delta(i)^{\otimes n} \otimes S^\lambda$, where $\lambda=\lambda^i$ is the only non-empty component. Then by \cite{ChTa} Lemma 3.5 part (3),
$$\mathrm{rad}(\Delta(\lambda))=\sum_{j=1}^n \Delta(i)^{\otimes j-1} \otimes \mathrm{Rad} \Delta(i) \otimes \Delta(i)^{\otimes n-j} \otimes S^\lambda.$$ Now if the radical of $\Delta(i)$ is a quotient of a direct sum of standard modules $M$, then there is a surjection of $A_{n-1} \otimes A_1$-modules
$$\Delta(i)^{n-1} \otimes M \otimes S^\lambda \rightarrow \Delta(i)^{n-1} \otimes \mathrm{Rad} \Delta(i) \otimes S^\lambda$$ which upon inducing and using the fact that as a $A_n$-module, $\mathrm{Rad}(\Delta(\lambda))$ is generated by the target, gives a surjection of $A_n$ modules
$$\mathrm{Ind}_{A_{n-1} \otimes A_1}^{A_n} \Delta(i)^{n-1} \otimes M \otimes S^\lambda \rightarrow \mathrm{Rad} \Delta(\lambda).$$ The domain here is a sum of standard modules by Lemma \ref{standard products}, proving this special case.

In general by using parts (3) and (4) of Lemma 3.5 from \cite{ChTa} we have
$$\mathrm{Rad}(\Delta(\lambda))=\mathrm{Ind}_{A_\lambda}^{A_n} \mathrm{Rad}( \bigotimes_{i \in I} \Delta(\lambda_i))=\mathrm{Ind}_{A_\lambda}^{A_n} (\sum_{i \in I} \bigotimes_{j \neq i} \Delta(\lambda_j) \otimes \mathrm{Rad}(\Delta(\lambda_i))).$$ Combining this equation with the special case already proved and Lemma \ref{standard products} finishes the proof.
\end{proof}

In the category $\OO$ for the Cherednik algebra of a symmetric group, there are certain combinatorially defined blocks called RoCK blocks which have the structure of a wreath product (see \cite{Tu1}, Definition 52).
\begin{corollary}
If $B_{w,\rho}\subset\OO_c(S_n)$ is a RoCK block of weight $w$ and $e$-core $\rho$ then $B_{w,\rho}$ has tight multiplicities.
\end{corollary}

\begin{proof}
$\OO_c(S_n)$ is equivalent as a highest weight category to $\mathcal{S}_q(n)-\mathrm{mod}$, the category of finite-dimensional modules over the $q$-Schur algebra \cite{Rouq}. Here $c=1/e$ with $2\leq e\leq n$, and $q$ is a primitive $e$'th root of $1$. Also, $\mathcal{S}_q(n)-\mathrm{mod}$ is highest-weight equivalent to the unipotent block of $GL(n,q)$ over a field of characteristic $\ell$ where $\ell>>0$ has order $e$ mod $q$ \cite{Ta}. Combining these facts with theorems proved by Miyachi and Turner, we have:
\begin{theorem}[\cite{Tu2},\cite{Mi1}]\label{Miyachi}
If $B_{w,\rho}$ is a RoCK block in $\OO_{1/e}(S_n)$ of weight $w$ and $e$-core $\rho$ then $B_{w,\rho}\simeq B^e_{1,\emp}\wr S_w$.
\end{theorem}
$B^e_{1,\emp}$ denotes a weight $1$ block, equivalent to the principal block of $\OO_{1/e}(S_e)$. It follows from the analysis in \cite{Rouq} that any weight $1$ block has tight multiplicities. Now the conclusion follows from Corollary \ref{wreath}.
\end{proof}

\begin{remark}
Call a block $B$ \textit{multiplicity-free} if $[\Del(\lambda):L(\mu)]\in\{0,1\}$ for all $\lambda,\mu\in B$. A block with tight multiplicities need not be multiplicity-free. The RoCK blocks provide an example of this: \cite{LeMi} and \cite{ChTa} discovered a formula for decomposition numbers for RoCK blocks in terms of Littlewood-Richardson coefficients, and any Littlewood-Richardson coefficient appears as a decomposition number in some RocK block, so RoCK blocks in general are not multiplicity-free.
\end{remark}

\begin{question} Does the Scopes equivalence class of RoCK blocks of weight $w$ exhaust all blocks of weight $w$ in $\bigoplus_n\OO_c(S_n)$ which have tight multiplicities? In other words, for type $A$ Cherednik algebras, is every block that has tight multiplicities highest weight equivalent to a RoCK block?
\end{question}

\begin{example} Outside type $A$, the principal block $B^0_{1/4}(E_6)$ of $\OO_{1/4}(E_6)$ is an example of a wreath product block which has tight multiplicities: $B^0_{1/4}(E_6)\simeq B^0_{1/4}(D_4)\wr S_3$ (where $B^0_{1/4}(D_4)$ denotes the principal block of $\OO_{1/4}(D_4)$) by results of Miyachi \cite{Mi2}. As mentioned in the introduction, for $D_4$ one should take $c_0=1/4$ and $d_0=d_1=0$, in which case the labels are $(0,0)$ in the first box, and $(1,-1)$ in the box with content $2$. Then $B^0_{1/4}(D_4)$ is $\ttt$-diagonalizable by Lemma \ref{diag} and has tight multiplicities by Theorem \ref{tight thm}. Therefore by Corollary \ref{wreath}, $B^0_{1/4}(E_6)$ has tight multiplicities.
\end{example}

\subsection{BGG algebras and the $\Ext^1$ quiver} For the rest of this section, let $C$ be a highest weight category with finite poset $\Lambda$ indexing the simple and standard objects.

We always have that $\dim\Ext^1(L(\lambda),L(\mu))$ is equal to the number of times $L(\mu)$ appears in the first radical layer of $P(\lambda)$, defined as $\rad_1 P(\lambda)=\Rad P(\lambda)/\Rad(\Rad P(\lambda))$. Suppose there is a contravariant duality functor $\delta:C\rightarrow C$, $\delta^2\simeq\mathrm{Id}$, such that $\delta(L)=L$ for any simple $L\in C$. Then we say that the highest weight category $C$ is \textit{BGG} \cite{Ir}. In this case, $\dim\Ext^1(L(\mu),L(\lambda))=\dim\Ext^1(L(\lambda),L(\mu))$ for any $\lambda,\mu\in\Lambda$, as can be seen by taking a nonsplit short exact sequence $0\rightarrow L(\mu)\rightarrow E \rightarrow L(\lambda) \rightarrow 0$ and applying $\delta$ to it. 

The $\Ext^1$ quiver of $C$ is defined to be the quiver with vertices $\lambda\in\Lambda$ and $\dim\Ext^1(L(\lambda),L(\mu))=:a_{\lambda,\mu}$ arrows $\lambda\rightarrow\mu$. Let us denote the $\Ext^1$ quiver for $C$ as $Q_{\Ext^1}(C)$. From the remarks above, if $C$ is BGG then $Q_{\Ext^1}(C)$ is a double quiver, that is, $a_{\lambda,\mu}=a_{\mu,\lambda}$.

\subsection{Primitive homomorphisms between standard modules}
Let $\lambda,\mu\in\Lambda.$ Call a homomorphism $\phi:\Del(\mu)\rightarrow\Del(\lambda)$ \textit{primitive} if it does not factor through any highest weight submodule $M$ of $\Del(\lambda)$ with highest weight $\nu$, $\mu<\nu<\lambda$. Call a homomorphism $\phi:\Del(\mu)\rightarrow\Del(\lambda)$ \textit{imprimitive} if there exists a highest weight submodule $M$ of $\Del(\lambda)$ of highest weight $\nu$, $\mu<\nu<\lambda$, such that a singular vector $v_\mu$ generating the image of $\phi$ belongs to $M$. Define $\Hom_{\mathrm{imprim}}(\Del(\mu),\Del(\lambda)$ to be the subspace of $\Hom(\Del(\mu),\Del(\lambda))$ generated by the imprimitive homomorphisms from $\Del(\mu)$ to $\Del(\lambda)$. Now define $\Hom_{\mathrm{prim}}(\Del(\mu),\Del(\lambda))$ to be the quotient of $\Hom(\Del(\mu),\Del(\lambda))$ by $\Hom_{\mathrm{imprim}}(\Del(\mu),\Del(\lambda))$. Note that if $\dim\Hom(\Del(\mu),\Del(\lambda))=1$ then $\Hom(\Del(\mu),\Del(\lambda))$ either consists entirely of primitive homomorphisms or imprimitive homomorphisms. If $\dim\Hom(\Del(\mu),\Del(\lambda))\in\{0,1\}$ for all $\mu,\lambda\in\Lambda$, then $\Hom_{\mathrm{prim}}(\Del(\mu),\Del(\lambda))$ or $\Hom_{\mathrm{imprim}}(\Del(\mu),\Del(\lambda))$ is either $0$ or simply the Hom space between the standards, any Hom space being the one or the other.

We may define another quiver $Q_{\mathrm{prim}}(B)$ for the block $B$: it has vertices $\lambda\in\Lambda$ and $b_{\mu,\lambda}$ arrows $\mu\rightarrow\lambda$ if $\dim\Hom_{\mathrm{prim}}(\Del(\mu),\Del(\lambda))=b_{\mu,\lambda}$.

\section{The Cherednik algebra}

Let $V$ be a finite dimensional $\CC$-vector space and $W \subseteq \mathrm{GL}(V)$ a finite subgroup of the group of linear transformations of $V$. A \emph{reflection} is an element $r \in W$ such that the codimension of the fix space of $r$ in $V$ is $1$. Let $R$ be the set of reflections in $W$, and for each $r \in R$ let $c_r \in \CC$ be a number such that $c_r=c_{w r w^{-1}}$ for all $r \in R$ and $w \in W$. We also fix a linear form $\alpha_r \in V^*$ with $$\mathrm{fix}(r)=\{\alpha_r=0 \}.$$ Nothing will depend upon our choice of $\alpha_r$, but the choice of $c_r$ is very important. We will write $c=(c_r)_{r \in R}$ for the collection of $c_r$'s.

Given a vector $y \in V$, the corresponding \emph{Dunkl operator} on $\CC[V]$ is given by the formula
$$y(f)=\partial_y(f)-\sum_{r \in R} c_r \la \alpha_r,y \ra \frac{f-r(f)}{\alpha_r} \quad \hbox{for $f \in \CC[V]$.}$$ The \emph{rational Cherednik algebra} is the subalgebra $H_c=H_c(W,V)$ of $\mathrm{End}_\CC(\CC[V])$ generated by $W$, $\CC[V]$, and the Dunkl operators $y$ for all $y \in V$.

\subsection{Category $\OO_c$} Category $\OO_c$ is the category of finitely generated $H_c$-modules on which each Dunkl operator $y \in V$ acts locally nilpotently. The objects of $\OO_c$ that can be constructed most directly are certain induced modules. Specifically, given a $\CC W$-module $U$ we inflate it to a $\CC[V^*] \rtimes W$-module by letting each $y \in V$ act by $0$. Then we put
$$\Delta_c(U)=\mathrm{Ind}_{\CC[V^*] \rtimes W}^{H_c} U.$$ If $\Lambda$ is an index set for the irreducible $\CC W$-modules, $\lambda \in \Lambda$, and $S^\lambda$ is the corresponding irreducible, then we will also write
$$\Delta_c(\lambda)=\Delta_c(S^\lambda).$$ These $\Delta_c(\lambda)$ are the \emph{standard modules} for $H_c$. Each carries a contravariant form $\la \cdot,\cdot \ra_c$ whose radical is the radical of the module $\Delta_c(\lambda)$. The quotient module
$$L_c(\lambda)=\Delta_c(\lambda)/\mathrm{Rad}(\la \cdot,\cdot \ra_c)$$ is irreducible, and as $\lambda$ runs over $\Lambda$ these give a complete set of non-isomorphic irreducible objects of $\OO_c$. 

\subsection{Partial orders on $\Irr W$}\label{partial}
There is a partial order $<_c$ on $\Irr W$ called the $c$-order with respect to which $\OO_c(W)$ is a highest weight category \cite{GGOR}. This partial order is defined using the action of the Euler element $h$, see section 3.1 of \cite{GGOR}: $h$ acts by a scalar $h_c(\lambda)$ on the highest weight $\lambda$ of any standard module $\Del(\lambda)$ in $\OO_c(W)$. We say that $\lambda>_c\mu$ if $h_c(\mu)-h_c(\lambda)\in\ZZ_{\geq 0}$. For the groups of type $G(r,1,n)$, we may define $h_c(\lambda)$ in terms of charged contents. Let $\tilde{c}(b)=rc_0\ct(b)+d_{\beta(b)}$ (see Section \ref{cyclostuff}) and $\hat{c}(\lambda)=\sum_{b\in\lambda}\tilde{c}(b)$. Then $$h_c(\lambda)=\frac{\dim\hh}{2}-\hat{c}(\lambda)$$
On the other hand, there is the natural partial order $<$ on $\Irr W$ given by the transitive closure of the relation: $\mu<\lambda$ if there is a nonzero homomorphism $\Del(\mu)\rightarrow\Del(\lambda)$. 

Note that Cherednik category $\OO_c(W)$ always satisfies BGG reciprocity \cite{GGOR}:
\begin{equation}\label{BGG}
[P(\lambda):\Del(\mu)]=[\Del(\mu):L(\lambda)]
\end{equation}
It is a consequence of BGG reciprocity that $\OO_c(W)$ has a unique coarsest partial order. This minimal partial order is given by $<$:

\begin{lemma}\label{ggjlpartial}
(\cite{GGJL}, Lemma 4.5) If $\prec$ coarsens the $c$-order $<_c$ and $\prec$ gives a highest weight structure on $\OO_c(W)$, then $\prec$ refines $<$.
\end{lemma}

\subsection{BGG properties of $\OO_c(W)$} 

Cherednik Category $\OO_c(W)$ for $W$ a real reflection group has a duality functor $\delta$ fixing simples, as described in \cite{GGOR}, Proposition 4.7 and Remark 4.9. Thus $\OO_c(W)$ is BGG when $W$ is real. Furthermore, if the parameter $c$ is real and $W$ is a complex reflection group, then $\OO_c(W)$ also has a duality functor, defined as follows. Fix a $W$-equivariant complex-antilinear isomorphism $\phi:\hh\rightarrow\hh^*$. If $c$ is real, then $\phi$ induces a complex-antilinear ring anti-involution $\bar{}:H_c(W)\rightarrow H_c(W)$ such that $\bar{x}=\phi^{-1}(x)$, $\bar{y}=\phi(y)$, and $\bar{w}=w^{-1}$ for all $x\in\hh^*$, $y\in\hh$, $w\in W$. Given a module $M$ in $\OO_c(W)$, let $DM$ be its graded complex-antilinear dual, with action given by $(hf)(m)=f(\bar{h}m)$ for $h\in H_c(W)$, $f\in DM$, $m\in M$. This gives a complex-antilinear additive exact functor $\delta:\OO_c(W)\rightarrow\OO_c(W)$, $\delta(M)=DM$. Then $\delta$ fixes simples. Thus $\OO_c(W)$ is also BGG so long as $c$ is real, which is the case in all examples considered in this paper.

\section{Orthogonal functions and tableau techniques for Cherednik algebra representation theory}\label{cyclostuff}
\subsection{The group $G(r,p,n)$.} Let $n$ and $r$ be positive integers. The group $G(r,1,n)$ consists of all $n$ by $n$ matrices with exactly one non-zero entry in each row and column, and such that the non-zero entries are all $r$'th roots of $1$. It acts on $V=\CC^n$ as a reflection group. Write $s_{ij}$ for the matrix interchanging the $i$'th and $j$'th coordinates and leaving the remaining coordinates fixed, let $\zeta=e^{2 \pi i /r}$ be a fixed primitive $r$'th root of $1$, and write $\zeta_i$ for the diagonal matrix with $1$'s on the diagonal except in position $i$ and with $i$'th diagonal entry equal to $\zeta$. 

Let $p$ be a positive integer dividing $r$. The group $G(r,p,n)$ is the subgroup of $G(r,1,n)$ consisting of matrices so that the product of the non-zero entries is an $r/p$'th root of $1$. The set of reflections in $G(r,p,n)$ is
$$R=\{\zeta_i^\ell s_{ij} \zeta_i^{-\ell} \ | \ 1 \leq i< j \leq n, \ 0 \leq \ell \leq r-1 \} \cup \{\zeta_i^{p\ell} \ | \ 1 \leq i \leq n, \ 1 \leq \ell \leq r/p-1 \}.$$

\subsection{The rational Cherednik algebra for $G(r,p,n)$.} We put $W=G(r,1,n)$. It is convenient to reparametrize $c$ in this case. We do so following the conventions in \cite{Gri}, (4.2) and (4.3). This parametrization may be specified by giving the defining relations for $H_c=H_c(W,V)$: it is the quotient of the algebra $T(V \oplus V^*) \rtimes W$ by the relations
$$[x_i,x_j]=0=[y_i,y_j] \ \hbox{for all $1 \leq i,j \leq n$,}$$  $$y_i x_j=x_j y_i+c_0 \sum_{0 \leq \ell \leq r-1} \zeta^{-\ell} \zeta_i^\ell s_{ij} \zeta_i^{-\ell} \quad \hbox{for $1 \leq i \neq j \leq n$,}$$ and
$$y_i x_i=x_i y_i  + 1-\sum_{0 \leq j \leq r-1} (d_j-d_{j-1}) e_{ij} - c_0 \sum_{\substack{j \neq i \\ 0 \leq \ell \leq r-1}} \zeta_i^\ell s_{ij} \zeta_i^{-\ell}  \quad \hbox{for $1 \leq i \leq n$,}$$ where 
$$e_{ij}=\frac{1}{r} \sum_{0 \leq \ell \leq r-1} \zeta^{-\ell j} \zeta_i^\ell.$$ The rational Cherednik algebra for $G(r,p,n)$ may be realized as the subalgebra of $H_c$ generated by $x_1,\dots,x_n,y_1,\dots,y_n$ and $G(r,p,n)$.

\subsection{A third commutative subalgebra of $H_c$.} The rational Cherednik algebra contains subalgebras $\CC[x_1,\dots,x_n]$ and $\CC[y_1,\dots,y_n]$ isomorphic to polynomial rings; the variables $x_i$ and $y_i$ act nilpotently on any finite dimensional representation. There is, however, a third polynomial subalgebra that often acts with simple spectrum. This is generated by the version of the Cherednik-Dunkl operators, also known as trigonometric Dunkl operators, introduced by Dunkl and Opdam in \cite{DuOp}. These are given by the formula
$$z_i=y_i x_i+c_0 \phi_i \quad \text{where} \quad \phi_i=\sum_{\substack{1 \leq j < i \\ 0 \leq \ell \leq r-1}} \zeta_i^\ell s_{ij} \zeta_i^{-\ell}.$$ We will write $\ttt=\CC[z_1,\dots,z_n,\zeta_1,\dots,\zeta_n]$ for the (commutative, though this is not obvious) subalgebra of $H_c$ generated by the Cherednik-Dunkl operators $z_i$ and the matrices $\zeta_i$. 

\subsection{$r$-partitions} We will use the following conventions for multi-partitions. A \emph{partition} is a weakly decreasing sequence $\lambda=(\lambda_1 \geq \lambda_2 \geq \cdots \geq \lambda_k)$ of non-negative integers. Given positive integers $r$ and $n$, an \emph{$r$-partition of $n$} is a sequence $\lambda=(\lambda^0,\lambda^1,\dots,\lambda^{r-1})$ of partitions $\lambda^i$ such that $n=\sum \lambda^i_j$. We will frequently visualize partitions via Young diagrams. For instance, the partition $(3,3,2)$ may be visualized as
$$\yng(3,3,2).$$ As is usual, we will identify two partitions that differ by a string of $0$'s, so that for instance $(3,3,2)=(3,3,2,0,0)$. The \emph{content} of a box in (the diagram of) a partition $\lambda$ is $$\mathrm{ct}(b)=\mathrm{col}(b)-\mathrm{row}(b),$$ where $\mathrm{col}(b)$ is $i$ if $b$ is in the $i$th column (counting from the left) and $\mathrm{row}(b)$ is $j$ if $b$ is in the $j$th row (counting from the top). Thus the list of contents of the boxes of $(3,3,2)$, read as in English from left to right and then top to bottom, is $(0,1,2,-1,0,1,-2,-1)$. For an $r$-partition $\lambda$ and a box $b \in \lambda$ we define $$\beta(b)=i \quad \text{if}  \quad b \in \lambda^i.$$

\subsection{When are two standard modules in the same block?}
The \emph{charged content} $c(b)$ of a box $b$ in an $r$-partition $\lambda$ is given by the formula
\begin{equation} \label{charged content}
c(b)=(d_{\beta(b)}-\beta(b))/r+ \mathrm{ct}(b) c_0
\end{equation} where if is a box of $\lambda^\ell$, $\mathrm{ct}(b)$ means the content of $b$ as a box of the ordinary partition $\lambda^\ell$. We define the \emph{$c$-weight} of a box $b$ to be the element of $\CC / \ZZ$ given by
$$w_c(b)=c(b) \ \text{mod} \ \ZZ,$$ and the \emph{$c$-weight} of $\lambda$ is the multiset of the $c$-weights of its boxes,
$$w_c(\lambda)=\{w_c(b)\;|\;b \in \lambda\}$$
 
Suppose the parameters $c$ are nonzero. By Theorem 2.11 of \cite{LyMa} and the Double Centralizer Theorem \cite{GGOR},
\begin{corollary}\label{block}
$\Delta_c(\lambda)$ and $\Delta_c(\mu)$ belong to the same block if and only if $w_c(\lambda)=w_c(\mu)$.
\end{corollary}
\begin{remark} The translation between Lyle-Mathas' criterion and ours is given by $e^{-2\pi iw_c(b)}=\mathrm{res}(b)$.
\end{remark} In the cases we study in this paper, when the denominator of $c_0$ is exactly $n$, we could use \cite{Gri} to characterize the $r$-partitions indexing standard modules in the principal block directly. 

\subsection{Tableaux and representations} Let $\lambda$ be an $r$-partition of $n$. A \emph{standard Young tableau on $\lambda$} is a filling $T$ of the boxes of $\lambda$ by the integers $1,2,\dots,n$ in such a way that the entries are increasing left to right and top to bottom within each component partition $\lambda^\ell$. We will denote the set of all standard Young tableaux on $\lambda$ by $\mathrm{SYT}(\lambda)$. When we write $T^{-1}(i)$ we will mean the box in $\lambda$ labeled by $i$.

The set of irreducible representations of $\CC G(r,1,n)$ is in bijection with the set of $r$-partitions of $n$ in such a way that if $S^\lambda$ is the irreducible $\CC G(r,1,n)$-module indexed by the $r$-partition $\lambda$, then there is a basis $v_T$ of $S^\lambda$ indexed by $T \in \mathrm{SYT}(\lambda)$  with
$$\zeta_i v_T=\zeta^{\beta(T^{-1}(i))} v_T \quad \text{and} \quad \phi_i v_T=r \mathrm{ct}(T^{-1}(i)) v_T \quad \hbox{for $1 \leq i \leq n$.} $$

\subsection{The $\ttt$-action on standard modules.} We reintroduce some of the notation from Section 2 of \cite{Gri}. Let $\alpha=(\alpha_1,\alpha_2,...,\alpha_n)\in\ZZ^n_{\geq 0}$. For $w\in S_n$, define a left action on $\ZZ^n_{\geq 0}$ by $$w\cdot\alpha=(\alpha_{w^{-1}(1)},\alpha_{w^{-1}(2)},...,\alpha_{w^{-1}(n)})$$
Let $w_\alpha\in S_n$ be the longest element (in Bruhat order) of all $w$ such that $w\cdot\alpha$ is a nondecreasing sequence. Finally, let $x^\alpha=x_1^{\alpha_1}x_2^{\alpha_2}...x_n^{\alpha_n}$. 

The standard module $\Del_c(\lambda)$ is isomorphic, as a vector space, to $\CC[\hh]\otimes \lambda$; as a basis of $\Del_c(\lambda)$ we may take

 $$\{x^\alpha w_\alpha^{-1} v_T\;|\;\alpha\in\ZZ^n_{\geq0},\;T\in\mathrm{SYT}(\lambda)\}$$ 

According to Theorem 5.1 of \cite{Gri} the action of $\ttt$ on this basis is given, for $1 \leq i \leq n$, by
$$\zeta_i x^\alpha w_\alpha^{-1} v_T=\zeta^{\beta(T^{-1} w_\alpha(i))-\alpha_i} x^\alpha w_\alpha^{-1} v_T$$ and
\begin{align*}
z_i x^\alpha w_\alpha^{-1} v_T&=\left(\alpha_i+1-(d_{\beta(T^{-1} w_\alpha(i))}-d_{\beta(T^{-1} w_\alpha(i))-\alpha_i-1})-r \mathrm{ct}(T^{-1} w_\alpha(i))c_0 \right) x^\alpha w_\alpha^{-1} v_T \\ &+\text{lower terms}.
\end{align*} 

\subsection{$\ttt$-semisimplicity of $\Delta_c(\lambda)$.} The standard module $\Delta_c(\lambda)$ for $H_c$ is, for generic $c$, diagonalizable with respect to $\ttt$. The precise result is as follows (see Lemma 7.1 from \cite{Gri}): for each component $\lambda^i$ of $\lambda$, define its \emph{diameter} $\mathrm{diam}(\lambda^i)$ to be the maximum difference between the contents of two of its boxes. Thus the diameter of $(3,3,2)$ is $4$. Assume $c_0 > 0$ (up to some obvious symmetries, this is the only really interesting case). Define $m$ to be the maximum over all components $\lambda^i$ of $\lambda$ that are not single columns of the diameter $\mathrm{diam}(\lambda^i)$. Then $\Delta_c(\lambda)$ is $\ttt$-diagonalizable if and only if $c_0$ is not a rational number with denominator at most $m$, and furthermore no equation of the form
$$k=c(b_1)-c(b_2) \quad \hbox{for $b_1 \in \lambda^i$, $b_2 \in \lambda^j$, $i \neq j$ and $k \in \ZZ$}$$ holds, where $c(b)$ denotes the charged content of $b$ as in \eqref{charged content}. In this situation, for each pair $(\alpha,T) \in \ZZ_{\geq 0}^n \times \mathrm{SYT}(\lambda)$ there is a unique $f_{\alpha,T} \in \Delta_c(\lambda)$ with $f_{\alpha,T}=x^\alpha w_\alpha^{-1} v_T+\text{lower terms}$ and
\begin{align*}
z_i f_{\alpha,T}&=\left(\alpha_i+1-(d_{\beta(T^{-1} w_\alpha(i))}-d_{\beta(T^{-1} w_\alpha(i))-\alpha_i-1})-r \mathrm{ct}(T^{-1} w_\alpha(i))c_0 \right)f_{\alpha,T}\\ \zeta_i f_{\alpha,T}&=\zeta^{\beta( T^{-1} w_\alpha(i))} f_{\alpha,T}
\end{align*}
for $1 \leq i \leq n$.

The above results imply the following lemma, giving a sufficient condition for every standard module in the principal block to be $\ttt$-diagonalizable.
\begin{lemma}\label{diag}
If $w_c(b) \neq w_c(b')$ for all $b \neq b'$ then every standard module in the same block as $\Delta_c(\lambda)$ is $\ttt$-diagonalizable. In particular, if $c_0$ is not a rational number of denominator at most $n-1$, then every standard module in the same block as $\Delta_c(\mathrm{triv})$ is $\ttt$-diagonalizable.
\end{lemma} 

\subsection{The submodule structure of $\Delta_c(\lambda)$ and tight multiplicities.} We will repeatedly make use Theorem 7.5 from \cite{Gri}. For the reader's convenience we include a statement of this result here. 

Assume $\Delta_c(\lambda)$ is $\ttt$-diagonalizable, so that the basis elements $f_{\alpha,T}$ are all well-defined. Given a box $b \in \lambda$ and a positive integer $k$, define a $\CC$-subspace $M_{b,k}$ by
$$M_{b,k}=\CC \{ f_{\alpha,T} \ | \ \alpha^-_{T(b)} \geq k \}.$$ Given two boxes $b_1,b_2 \in \lambda$ and a positive integer $k$, let
$$M_{b_1,b_2,k}=\CC \{ f_{\alpha,T} \ | \ \hbox{ $ \alpha^-_{T^{-1}(b_1)} - \alpha^-_{T^{-1}(b_2)} \geq k$, with equality implying $w_\alpha^{-1}(T(b_1)) < w_\alpha^{-1}(T(b_2))$} \}.$$ Now we state our main tool, Theorem 7.5 from \cite{Gri} (which is a generalization of Theorem ??? from \cite{Gri2}).
\begin{theorem} \label{submodule structure}
Suppose $\Delta_c(\lambda)$ is $\ttt$-diagonalizable. Then the lattice of submodules of $\Delta_c(\lambda)$ is generated by those of following two forms:
\begin{itemize} 
\item[(1)] $M_{b,k}$, for a box $b \in \lambda$ and a positive integer $k$ such that $k=r c(b)+\beta(b)-d_{\beta(b)-k}$;
\item[(2)] $M_{b_1,b_2,k}$, for boxes $b_1,b_2 \in \lambda$ and a positive integer $k$ with $k=r c(b_1)+\beta(b_1)-r c(b_2)-\beta(b_2) \pm r c_0 $ and $k=\beta(b_1)-\beta(b_2) \ \mathrm{mod} \ r$.
\end{itemize}
\end{theorem}

We will refer to the submodules $M_{b,k}$ and $M_{b_1,b_2,k}$ appearing in Theorem \ref{submodule structure} as \emph{fundamental submodules}. 

The lowest degree subspace of a given submodule consists of singular vectors. Thus for each fundamental submodule we obtain a map from a standard module to $\Delta_c(\lambda)$. We can specify which standard modules map into $\Delta_c(\lambda)$ this way by examining the $\ttt$-eigenvalues of the $f_{\alpha,T}$'s spanning the lowest degree part.
\begin{corollary}\label{lowest degree subspace}
Suppose $\Delta_c(\lambda)$ is diagonalizable and that $c_0>0$.
\begin{itemize}
\item[(a)] In case (1) from Theorem \ref{submodule structure}, the isotype of the lowest degree subspace of $M_{b,k}$ is the module $S^\mu$ where $\mu$ is obtained from $\lambda$ by removing the subdiagram of $\lambda^{\beta(b)}$ consisting of $b$ and all boxes (weakly) below and to the right of it, and setting $\mu^{\beta(b)-k}$ equal to this subdiagram, leaving all other components of $\lambda$ unchanged. 

\item[(b)] In case (2) from Theorem \ref{submodule structure}, when the plus sign holds in the term $+rc_0$ (resp., the minus sign holds in the term $-rc_0$), the isotype of the lowest degree space of $M_{b_1,b_2,k}$ is the module $S^\mu$, where $\mu$ is obtained from $\lambda$ by removing the subdiagram of $\lambda^{\beta(b_1)}$ consisting of $b_1$ and all boxes below it (resp., to its right), and attaching it to $\lambda^{\beta(b_2)}$ in such a way that $b_1$ is the box directly below $b_2$ (resp., to its right).
\end{itemize}
\end{corollary}

The next corollary is a sufficient condition for the principal block of $\OO_c$ to have tight multiplicities that is practical to verify in examples.

\begin{corollary} \label{principal tight corollary} Suppose that every standard module in the principal block is diagonalizable and that the following condition holds: for every standard module and every submodule $M$ of the form $M_{b,k}$ or $M_{b_1,b_2,k}$ as in the previous theorem, if the lowest degree subspace of $M$ is contained in a module of the form (2), then $M$ is as well. Then the principal block has tight multiplicities. \end{corollary}
\begin{proof}
If $M$ and $N$ are submodules that are both generated by singular vectors, then so is their sum $M+N$. So it suffices to prove that submodule of the forms (1) and (2) is generated by singular vectors. Let $M$ be such a submodule and let $N \subseteq M$ be the submodule generated by the singular vectors in $M$. By Theorem \ref{submodule structure} $N$ is in the lattice of submodules, so equal to a sum of intersections of modules of the forms (1) and (2). Evidently $N$ contains the lowest degree piece $S$ of $M$. By our hypothesis $S$ is only contained in a submodule of the form (2) if $M$ is. On the other hand, it follows from the definitions that if $S$ is contained in a submodule of the form (1) then so is $M$. Therefore every submodule of the form (1) or (2) that contains $S$ contains all of $M$, and hence $N=M$ as desired.
\end{proof}

\subsection{Graded dimension of $L_c(\Triv)$}

\begin{corollary} \label{dim formula}
Suppose $c_0=\ell/n$ for a positive integer $\ell$ coprime to $n$. Then the set of $f_\alpha$ such that
\begin{itemize}
\item[(a)] we have $\alpha_n^--\alpha_1^- \leq \ell r$ with equality implying $w_\alpha^{-1}(n) > w_\alpha^{-1}(1)$, and 
\item[(b)] for each pair of integers $k$ and $m$ with $0 \leq m \leq n-1$ and $k >0$ such that $$d_0-d_{-k}+rm \ell/n=k,$$ we have $\alpha_{m+1}^- < k$
\end{itemize} is a homogeneous basis of $L_c(\Triv)$. In particular the graded dimension of $L=L_c(\Triv)$ for $G(2,2,2n)$ with $c=1/2n$ is
$$\sum \mathrm{dim}_\CC(L^d)q^d= \sum_{\substack{0 \leq \ell \leq n-1 \\ n-1-\ell \in 2 \ZZ}} {2n \choose \ell} q^{n-1} + \sum_{d=0}^{n-2} \left(\sum_{\substack{0 \leq \ell \leq d \\ d-\ell \in 2 \ZZ}}  {2n \choose \ell} \right) (q^d+q^{2(n-1)-d}),$$ and the graded dimension of $L=L_c(\Triv)$ for $G(2,1,2n)$ with $c=1/2n$ is
$$ \sum \mathrm{dim}_\CC(L^d)q^d= \sum_{\substack{0 \leq \ell \leq n \\ n-\ell \in 2 \ZZ}} {2n \choose \ell} q^n + \sum_{d=0}^{n-1} \left(\sum_{\substack{0 \leq \ell \leq d \\ d-\ell \in 2 \ZZ}}  {2n \choose \ell} \right) (q^d+q^{2n-d}).$$
\end{corollary}
\begin{proof}
That the given set of $f_\alpha$'s give a basis of $L_c(\Triv)$ follows immediately from Theorem \ref{submodule structure}. In the case of $G(2,1,n)$ with the parameters $c_0=1/2n$ and $d_0=0=d_1$, this specializes to show that
\begin{equation} \label{Dtriv basis} \{f_\alpha \ | \ \alpha^-_{n+1}=0 \ \text{and} \ \alpha^-_{2n}-\alpha^-_1 \leq 2 \ \hbox{with equality implying $w_\alpha^{-1}(2n) > w_\alpha^{-1}(1)$}\}\end{equation} is a basis of $L=L_c(\Triv)$. Moreover, since $d_0=0$, the Cherednik algebra of type $G(2,2,n)$ is a subalgebra of that of type $G(2,1,n)$ and the module $L$ restricts to this subalgebra to give the spherical irreducible.

The condition $\alpha^-_{n+1}=0$ means that any $\alpha$ indexing a basis element of $L$ must have at least $n+1$ $0$'s, or equivalently, at most $n-1$ non-zero entries. In the presence of this constraint, the condition $\alpha^-_{2n}-\alpha^-_1 \leq 2$ implies that the largest entry in $\alpha$ is at most $2$, and finally the condition $w_\alpha^{-1}(2n) > w_\alpha^{-1}(1)$ if some $2$ appears means that all the zeros appear to the left of all the twos. Thus the top degree piece occurs in dimension $2(n-1)$, in accordance with the claimed graded dimension formula. By symmetry it suffices to prove that the dimension of each graded piece in degree at most $n-1$ is as claimed.

The set of possible $\alpha^-$'s that can arise is
$$\{(0^k, 1^\ell ,2^m \ | \ k \geq n+1 \ \text{and} \ k+\ell+m=2n \}.$$ Given such a $\alpha^-$, the conditions in \eqref{Dtriv basis} imply that choosing a $\alpha$ rearranging to $\alpha^-$ and corresponding to a basis element of $L$ is equivalent to choosing the position of the $\ell$ $1$'s, which may be done in exactly ${2n \choose \ell}$ ways. It follows that, for $d \leq n-1$, the dimension of the polynomial degree $d$ piece of $L$ is
$$\mathrm{dim}_\CC(L^d)=\sum_{\substack{0 \leq \ell \leq d \\ d-\ell \in 2 \ZZ}}  {2n \choose \ell}$$
in which the $\ell$'th summand comes from $\alpha$'s with $\ell$ $1$'s and $(d-\ell)/2$ $2$'s. This proves the type $G(2,2,2n)$ formula, and the type $G(2,1,2n)$ case is entirely analogous. \end{proof}

These formulas together with some elementary manipulatorics prove the conjectures of Oblomkov-Yun mentioned in the introduction.

The Oblomkov-Yun conjecture for $\dim L(B_{2n})$ concerns the Cherednik algebra of type $G(2,1,2n)$ at equal parameters $1/2n$; generalizing to the Cherednik algebra of $G(r,1,rn)$ at equal parameters $1/rn$, it turns out that for $r>1$ there is a natural dimension formula for $L(\Triv)$ that specializes to the dimension formula for $L(B_{2n})$ when $r=2$. Take equal parameters $c_0=c_1=...=c_{r-1}=1/rn$ for $H_c(G(r,1,rn))$. Let $\zeta$ be a primitive $r$'th root of $1$. Then for each $0\leq k\leq r-1$, 
$$d_k=\sum_{1\leq \ell\leq r-1}\zeta^{k\ell} c_\ell=\begin{cases}(r-1)/rn & \hbox{ if }k=0\\ -1/rn & \hbox{ if }1\leq k \leq r-1\end{cases}$$
We have $L(\Triv)=L((rn),\emp,...,\emp)$. The $r$-version of the dimension formula for $L(B_{2n})$ is straightforward to prove, using Theorem \ref{submodule structure} and generalizing the argument in Corollary \ref{dim formula}:
\begin{corollary} \label{rdim}
\begin{enumerate}
\item The radical of $\Del(\Triv)$ is the sum of the following $r$ submodules:
\begin{itemize}
\item For each $k=1,...,r-1$, there is a submodule $M_{b,k}$ where $b$ is the $kn$'th box of $\Triv$.
\item There is a single submodule of the form $M_{b_1,b_2,r}$ by taking $b_1$ and $b_2$ to be the rightmost and leftmost boxes, respectively, of $\Triv$.
\end{itemize}
\item $L(\Triv)$ has the basis
$$\{ f_\alpha \ | \ \alpha_{kn}^-\leq k-1 \ \hbox{for each} \ 1\leq k\leq r-1,\ \alpha_{rn}^-\leq r \ \hbox{with equality implying} \ w_\alpha^{-1}(rn)>w_\alpha^{-1}(1)\}$$
\item
\begin{align*}
\dim L(\Triv)=&\sum_{j_r=0}^n\sum_{j_{r-1}=0}^{n-j_r}{rn\choose j_{r-1}}\sum_{j_{r-2}=0}^{2n-j_r-j_{r-1}}{rn-j_{r-1}\choose j_{r-2}}\sum_{j_{r-3}=0}^{3n-j_r-j_{r-1}-j_{r-2}}{rn-j_{r-1}-j_{r-2}\choose j_{r-3}}...\\
&\qquad\qquad\qquad\qquad...\sum_{j_1}^{(r-1)n-j_r-...-j_2}{rn-j_{r-1}-...-j_2\choose j_1}
\end{align*}
\end{enumerate}
\end{corollary}

\subsection{Coordinate rings of subspace arrangements} In many cases of interest, a quotient of the polynomial represnetation $\Delta_c(\Triv)$ may be identified with the coordinate ring of a subscheme of $V$, and information about the subscheme obtained from representation theory. For instance, the coordinate ring of the set of points in $\CC^n$ having at least $k$ coordinates equal to one another is the quotient of the polynomial representation of the type $A_n$ Cherednik algebra at parameter $c=1/k$ by its socle, which is a unitary representation. A conjectural BGG resolution (see \cite{BGS}) of this unitary representation is then a minimal resolution of the coordinate ring of the $k$-equals arrangement. The general principle allows the calculation of data of interest in commutative algebra in terms of the combinatorics used by representation theorists.

We can combine Theorem \ref{submodule structure} with Lemma 2.9 from \cite{BGS} to give a similar description of $L_c(\Triv)$ for types $G(2,1,2n)$ and $G(2,2,2n)$.
\begin{corollary}
The module $L_c(\Triv)$ for $G(2,2,2n)$ at parameter $c=1/2n$ is the coordinate ring of the scheme theoretic intersection of the set of points in $\CC^{2n}$ with $x_1^2=x_2^2=\cdots=x_{2n}^2$ with the set of points having at least $n+1$ coordinates equal to zero. Likewise, the module $L_c(\Triv)$ for $G(2,1,2n)$ at parameter $c=1/2n$ is the coordinate ring of the scheme theoretic intersection of the set of points in $\CC^{2n}$ with $x_1^2=x_2^2=\cdots=x_{2n}^2$ with the set of points having at least $n$ coordinates equal to zero.
\end{corollary}
\begin{proof}
In the $G(2,2,2n)$ case, Lemma 2.9 from \cite{BGS} implies that the two ideal of $x_1^2=\cdots=x_{2n}^2$ and the ideal of the set of points having at least $n+1$ coordinates equal to zero are both submodules. Now Theorem 4.4 implies that these must be equal to the submodules $M_{b_1,b_2,2}$ (where $b_1$ is the box of content $2n-1$ and $b_2$ is the box of content zero in $\Triv$) and $M_{b,1}$ (where $b$ is the box of content n), respsectively. Therefore the radical of $\Delta_c(\Triv)$ is the sum of these two ideals and the result follows. The case $G(2,1,2n)$ is analogous.
\end{proof}

\subsection{Jantzen filtration} In this subsection we explain how to explicitly describe the Jantzen filtration on $\Delta_c(\lambda)$ in case it is $\ttt$-diagonalizable. 

We will write $H(G,V)$ and $\Delta(\lambda)$ for the Cherednik algebra and its standard modules assuming the parameters $c$ are polynomial variables. The contravariant form $\la \cdot,\cdot \ra$ on $\Delta(\lambda)$ then takes value in the ring $A=\CC[c_r]_{r \in R}$ of polynomials in the parameters. Fixing a numerical parameter $c$, we write $m_c \subseteq A$ for the corresponding maximal ideal, so that we have quotient maps $H(G,V) \rightarrow H_c(G,V)=H(G,V)/m_cH(G,V)$ and $\pi:\Delta(\lambda) \rightarrow \Delta_c(\lambda)=\Delta(\lambda)/m_c \Delta(\lambda)$. 

Define a filtration on $\Delta(\lambda)$ by 
$$\Delta(\lambda)^{\geq d}=\{f \in \Delta(\lambda) \ | \ \la f,g \ra \in I_c^d \ \hbox{for all $g \in \Delta(\lambda)$.} \}.$$ The \emph{Jantzen filtration} on $\Delta_c(\lambda)$ is the image of this filtration by the quotient map,
$$\Delta_c(\lambda)^{\geq d}=\pi(\Delta(\lambda)^{\geq d}).$$ We note that we may also compute the Jantzen filtration by first localizing $A$, replacing it in the definitions by the local ring $A_c$ at $m_c$ before specializing to $A/m_c$. If $c$ is a parameter such that the polynomials $f_{\alpha,T}$ are well-defined at $c$, these may also be regarded as elements of the standard module localized at $m_c$.

The next lemma describes an explicit basis of the $d$th submodule in the Jantzen filtration in terms of the fundamental submodules defined following Theorem \ref{submodule structure}.
\begin{lemma} \label{Jantzen lemma}
Assume that $\Delta_c(\lambda)$ is $\ttt$-diagonalizable. We have
$$\Delta_c(\lambda)^{\geq d}=\CC \{f_{\alpha,T} \ | \ \hbox{$f_{\alpha,T}$ belongs to at least $d$ fundamental submodules.} \}$$
\end{lemma}
\begin{proof}
Each $f_{\alpha,T} \in A_c \otimes_A \Delta(\lambda)$ may be constructed by applying the intertwining operators $\sigma_i$ and $\Phi$ in some sequence (depending on $(\alpha,T))$ to a certain basis element $w_0 v_T \in S^\lambda$, in such a way that once we enter a fundamental submodule we never leave it. The norm $\la f_{\alpha,T},f_{\alpha,T} \ra$ is a rational function of $c$ which may be written as a product of linear factors divided by another product of linear factors using the recursions in the proof of Theorem 6.1 from \cite{Gri}. Examining these recursions shows that the denominators are never zero at $c$, and the numerator acquires precisely one zero every time we enter some fundamental submodule.
\end{proof}

\subsection{Tableaux indexation for $f_{\alpha,T}$} Given $T \in \mathrm{SYT}(\lambda)$ and $\alpha \in \ZZ_{\geq 0}^n$, define fillings $P$ and $Q$ of the boxes of $\lambda$ by the rules
$$P(b)=w_\alpha^{-1} (T(b)) \quad \text{and} \quad Q(b)=\alpha_{P(b)} \quad \hbox{for all $b \in \lambda$.}$$ We may recover $\alpha$ and $T$ from $P$ and $Q$ by the rules
$$\alpha_i=Q(P^{-1}(i)) \quad \text{and} \quad T(b)=w_\alpha(P(b)).$$ These mappings define inverse bijections between $\ZZ_{\geq 0}^n \times \mathrm{SYT}(\lambda)$ and the set of pairs $(P,Q)$, where $P$ is a bijection from the boxes of $\lambda$ to the set $\{1,2,\dots,n\}$, $Q$ is a filling of the boxes of $\lambda$ by non-negative integers, weakly increasing left to right and top to bottom, and we have
$$P(b) > P(b') \quad \hbox{whenever $b < b'$ and $Q(b)=Q(b')$.}$$ Here we have defined an ordering on the boxes by $b< b'$ if $T(b) < T(b')$ for all standard Young tableaux $T$ (in other words, if $b$ is up and to the left of $b'$ in the same component of $\lambda$). We will write $\gimel(\lambda)$ for this set of pairs $(P,Q)$; it is identified via the above bijections with the set of vertices of the \emph{calibration graph} defined in \cite{Gri}.

\subsection{Indexation of the irreducibles in the principal block}\label{indexation} Given $0 \leq i \leq r-1$ and a box $b \in (n)$ for which the equation
$$(d_i-i)/r=d_0/r+\mathrm{ct}(b) c_0+k_i$$ holds for some integer $k_i$, we put $b_i=b$. Note that since the denominator of $c_0$ is $n$, the box $b$ and the integer $k_i$ are uniquely determined by $i$ if they exist. This produces a list of boxes $b_i$, one for each $0\leq i \leq r-1$ for which the equation has a solution, each such box being accompanied by the pair of integers $(i,k_i)$. Notice that we may well have $b_i=b_j$ for some $i \neq j$. We can record these data as in the introduction, by placing a label $(i,k_i)$ in the box $b_i$. We repeat the example from the introduction here:
$$
\ytableausetup{mathmode, boxsize=3.5em}
\begin{ytableau}
\begin{matrix} (0,0) \\ (1,-2) \end{matrix} & (3,-1) & & & & \begin{matrix} (5,3) \\ (2,-2) \end{matrix} &  & 
\end{ytableau}
$$
Given a multipartition $\lambda$ with $w_c(\lambda)=w_c(\Triv)$ we define a pair of tableaux on $\Triv$: first, we observe that since the denominator of $c_0$ is exactly $n$, there is a unique box $\overline{b}$ of $\lambda$ for each box $b$ of $\Triv$ satisfying
$$c(\overline{b})=c(b) \ \mathrm{mod} \ \ZZ.$$ We then define two fillings of the boxes of $(n)$ as follows: for the first, we put $S(b)=\beta(\overline{b})$ in the box $b$, and for the second, we place the integer $T(b)$ defined by the equation
$$c(\overline{b})=c(b)+T(b)$$ in the box $b$. We write $S(\lambda)$ and $T(\lambda)$ for these tableaux. It is not difficult to see that $T(\lambda)$ is uniquely determined from $S(\lambda)$ and the list $(b_i,k_i)$ unless $S(\lambda)$ contains only one integer $i$ (that is, unless $\lambda^i=\emptyset$ for all but one index $i$). 

This defines a bijection from the set of $r$-partitions $\lambda$ with $w_c(\lambda)=w_c(\Triv)$ onto the set of pairs $(S,T)$ of tableaux on $(n)$ with the following properties. First, we identify the right-hand border of the last box of $(n)$ with the left-hand border of the first box of $(n)$ so that the set of boxes becomes circular. Then the set of boxes $b$ with $S(b)=i$ is an interval (in this circular version of $(n)$) containing the  box $b_i$ labeled $(i,k_i)$ (and empty if there is no box $b_i$), we have $T(b_i)=k_i$, and $T(b)=k_i$ or $T(b)=k_i \pm \ell$ for all boxes $b$ with $S(b)=i$, with $T(b)=k_i+\ell$ if when traveling left to right from $b_i$ to $b$ we cross from the $n$th box to the $1$st.

For the example of the block labeling given above and in the introduction and with $c_0=3/8$, here are possible pairs $(S,T)$:
$$
S=\ytableausetup{mathmode, boxsize=2em}
\begin{ytableau}
 0 & 0 & 5 & 5 & 5 & 5 & 5 & 0
\end{ytableau}
$$ 
$$
T=\ytableausetup{mathmode, boxsize=2em}
\begin{ytableau}
 0 & 0 & 3 & 3 & 3 & 3 & 3 & -3
\end{ytableau}
$$ corresponds to
$$
\lambda=\left(\ytableausetup{mathmode, boxsize=1em}
\begin{ytableau} \hfil & \hfil \\ \hfil \end{ytableau}, \emptyset,\emptyset,\emptyset,\emptyset, \begin{ytableau} \hfil & \hfil \\ \hfil \\ \hfil \\ \hfil  \end{ytableau}\right)
$$ while
$$
S=\ytableausetup{mathmode, boxsize=2em}
\begin{ytableau}
 2 & 3 & 3 & 3 & 2 & 2 & 2 & 2
\end{ytableau}
$$ 
$$
T=\ytableausetup{mathmode, boxsize=2em}
\begin{ytableau}
 1 & -1 & -1 & -1 & -2 & -2 & -2 & -2
\end{ytableau}
$$ corresponds to
$$
\lambda=\left(\ytableausetup{mathmode, boxsize=1em}
 \emptyset,\emptyset,\begin{ytableau} \hfil & \hfil & \hfil &Â \hfil \\ \hfil \end{ytableau}, \begin{ytableau} \hfil & \hfil & \hfil  \end{ytableau} ,\emptyset,\emptyset \right)
$$

\subsection{Submodules and maps} \label{submodules}

In terms of the previous parametrization, the submodule structure is determined using Theorem \ref{submodule structure} as follows: there is a submodule of the form $M_{b,k}$ whenever $b$ is labeled by $(i,k_i)$ in the block tableau and $S(b)=j \neq i$ while $(j,T(b))>(i,k_i)$, with $k=r(T(b)-k_i)+j-i$. There is a submodule of the form $M_{b_1,b_2,k}$ for each pair of adjacent boxes $b_1,b_2$ with  $(S(b_1),T(b_1))>(S(b_2),T(b_2))$ where $k=r(T(b_1)-T(b_2))+S(b_1)-S(b_2)$.  In addition, if $b_1$ is the last box of $(n)$ and $b_2$ is the first box, and we have $(S(b_1),T(b_1)) > (S(b_2),T(b_2)-\ell)$ then $M_{b_1,b_2,k}$ is a submodule with $k=r(T(b_1)-T(b_2)+\ell)+S(b_1)-S(b_2)$.

Thus for the first pair $(S,T)$ above, and labeling the boxes as follows
$$
\ytableausetup{mathmode, boxsize=2em}
\begin{ytableau}
 b_3 & b_2 & b_1 & \hfil & \hfil & b_6 & b_4 & b_5
\end{ytableau}
$$ we have submodules $M_{b_1,b_2,23}$, $M_{b_3,11}$, $M_{b_4,b_5,41}$, and $M_{b_6,33}$. 


\subsection{Tight multiplicities}\label{tight subsection}

Given pairs of integers $(i_1,k_1)$ and $(i_2,k_2)$ we write $(i_1,k_1)>(i_2,k_2)$ if $k_1 > k_2$ or $k_1=k_2$ and $i_1>i_2$. 

\begin{theorem}\label{tight thm} Suppose $c_0=\ell/n$ for a positive integer $\ell$ coprime to $n$. Then the principal block has tight multiplicities if and only if the following condition holds: read from left to right, and breaking ties between labels in the same box by reading $(i,k)$ before $(j,k')$ if $(i,k)>(j,k')$, in the block tableau, if the labels are $((i_1,k_1),\dots,(i_s,k_s)$ then we have
$$(i_1,k_1) > (i_2,k_2)> \cdots > (i_s,k_s) > (i_1,k_1-\ell).$$
\end{theorem}
\begin{proof}
Suppose first that the condition 
$$(i_1,k_1) > (i_2,k_2)> \cdots > (i_s,k_s) > (i_1,k_1-\ell)$$ on the labels read left to right holds.  It then  follows from the translation of Theorem \ref{submodule structure} in \ref{submodules} that:
\begin{itemize}
\item[(1)] Whenever submodules $M_{b,k}$ and $M_{b_1,b_2,k'}$ both appear with $\beta(b)=\beta(b')$, then we have $k<k'$ and hence the lowest degree space of $M_{b,k}$ is not contained in $M_{b_1,b_2,k'}$. 
\item[(2)] Whenever submodules $M_{b_1,b_2,k}$ and $M_{b_1',b_2',k'}$ appear, we have $\beta(b_1) \neq \beta(b_1')$. 
\end{itemize}
Now these conditions imply that the hypotheses of Corollary \ref{principal tight corollary} are satisfied, so the radical of $\Delta_c(\lambda)$ is generated by singular vectors. 

The case in which $S(b)=i$ for all $b$ is similar but easier, as only one submodule of the form $M_{b_1,b_2,k}$ can appear. This proves the sufficiency of our condition.

If a standard module $\Delta_c(\lambda)$ has tight multiplicities, then its composition length is equal to the number of irreducible $G(r,1,n)$-modules, counted with multiplicities, appearing in the space of singular vectors. On the other hand, this composition length is, by Theorem 7.1 of \cite{Gri}, equal to the number of distinct intersections of the submodules of types (1) and (2) appearing in Theorem \ref{submodule structure}. Now we have a map from irreducible $G(r,1,n)$-submodules of the space of singular vectors to submodules of $\Delta_c(\lambda)$ given by $S \mapsto \CC[V] S$, and evidently $\CC[V] S$ has lowest degree subspace $S$ so this map is an injection. Again using Theorem \ref{submodule structure}, by irreducibility of $S$ the submodule $\CC[V] S$ of $\Delta_c(\lambda)$ must be equal to an intersection of submodules of the forms (1) and (2). It now follows by comparing cardinalities that if $\Delta_c(\lambda)$ has tight multiplicities, then every such intersection is of the form $\CC[V] S$ for some irreducible $G(r,1,n)$-submodule $S$ of the space of singular vectors, and that distinct intersections have different lowest degree parts $S$.

We first suppose that some label $(i,k)$ appearing in the block tableau has $(i,k)<(i_1,k_1-\ell)$. If the box $b$ has the label $(i,k)$, then setting $\lambda$ equal to the multiparition with $\lambda^i=(n)$ and all other $\lambda^j$ empty, the standard module $\Delta_c(\lambda)$ has a submodule $M_{b,r(k_1-k)+(i_1-i)}$ with 
$$r(k_1-k)+(i_1-i) > r \ell.$$ If $b_1$ is the last box of $(n)$ and $b_1$ is the first, then $\Delta_c(\lambda)$ also has the submodule $M_{b_1,b_2,r \ell}$. But the lowest degree piece of $M_{b,r(k_1-k)+(i_1-i)}$ is contained in $M_{b_1,b_2,r \ell}$ without having a containment of
$M_{b,r(k_1-k)+(i_1-i)}$ in $M_{b_1,b_2,r \ell}$, so the distinct intersections $M_{b,r(k_1-k)+(i_1-i)}$ and $M_{b,r(k_1-k)+(i_1-i)} \cap M_{b_1,b_2,r \ell}$ have the same lowest degree space. By the argument of the previous paragraph $\Delta_c(\lambda)$ does not have tight multiplicities. 

Now suppose that there are labels $(i,k)$ and $(j,m)$ with $(i,k) < (j,m)$ and $(i,k)$ appearing strictly to the left of $(j,m)$ in the block tableau. Let $\lambda$ be the multipartition with $\lambda^j=(n)$ and all other components empty. Let $b$ be the box of $\lambda^j$ determined by $c(b)=c(b')$ mod $\ZZ$, where $b'$ is the box of $(n)$ labeled by $(i,k)$, let $b_1$ be the rightmost box of $\lambda^j$ and let $b_2$ be the leftmost box. We then have submodules 
$$M=M_{b_1,b_2,r \ell} \quad \text{and} \quad N=M_{b,r(m-k+\ell)+j-i}$$ with the lowest degree piece of $N$ contained in $M$ but $N$ not contained in $M$. We conclude as in the previous paragraph.
 
\end{proof}

\subsection{The quiver of primitive homs and tight multiplicities}
We have the quiver $Q_{\mathrm{prim}}(B_0)$ for the principal block $B_0$ of  $\OO_c(G(r,1,n))$ at $c_0=\ell/n$: it has vertices $\lambda\in\Lambda$ and an arrow $\mu\rightarrow\lambda$ if $\dim\Hom_{\mathrm{prim}}(\Del(\mu),\Del(\lambda))=1$. (Recall that all Hom spaces are $0$ or $1$-dimensional in $B_0$). We will state a criterion in terms of quivers for detecting whether $B_0$ has tight multiplicities.

We will need a graded version of BGG reciprocity to hold for $B_0$.
\begin{lemma}\label{gr bgg} Let $\lambda,\mu\in\OO_c(G(r,1,n))$. Then:
$$[P(\lambda):L(\mu)](v)=\sum_{\sigma\in\Lambda}[P(\lambda):\Del(\sigma)](v)[\Del(\sigma):L(\mu)](v)=\sum_{\sigma\in\Lambda}[\Del(\sigma):L(\lambda)](v)[\Del(\sigma):L(\mu)](v)$$
\end{lemma}
\begin{proof}
There is an equivalence of $\OO_c(G(r,1,n))$ with a block of type $A$ affine parabolic category $\OO$ \cite{RSVV}. A graded version of the latter category is Koszul \cite{RSVV}. The argument of Beilinson-Bernstein \cite{BB} implies that the graded lift of a standard object coincides with the Jantzen filtration on that standard. Proposition 1.8 of \cite{Shan} implies that the Jantzen filtration on a standard in $\OO_c(G(r,1,n))$ goes to the Jantzen filtration on a standard in the type $A$ affine parabolic category $\OO$ under the equivalence of \cite{RSVV}. This implies the statement.
\end{proof}

\begin{theorem}\label{quiver theorem}
Suppose $B$ is a block of a highest weight category with finitely many simple objects, which is BGG and satisfies a graded BGG reciprocity rule as in Lemma \ref{gr bgg}. Then $B$ has tight multiplicities if and only if the $\Ext^1$ quiver $Q_{\Ext^1}(B)$ is the double of the quiver $Q_{\mathrm{prim}}(B)$.
\end{theorem}
\begin{proof}
Since $B$ is BGG, $Q_{\Ext^1}(B)$ is guaranteed to be a double quiver, and it suffices to consider the number of arrows $\mu\rightarrow\lambda$ in $Q_{\Ext^1}(B)$ for $\mu<\lambda$. First, let us establish that the information we need about $\Ext^1$ is contained in the structure of $\Del(\lambda)$. We know that $\dim\Ext^1(L(\lambda),L(\mu))=[\rad_1 P(\lambda):L(\mu)]$, and the latter is equal to the coefficient of $v$ in the graded decomposition number $[P(\lambda):L(\mu)](v)$. Now we use the graded BGG reciprocity rule as in Lemma \ref{gr bgg}. If $\sigma\neq\lambda,\mu$ then $[\Del(\sigma):L(\lambda)](v)$ and $[\Del(\sigma):L(\mu)](v)$ both belong to $v\NN[v]$; also, $[\Del(\mu):L(\lambda)](v)=0$ because $\mu<\lambda$. Thus a single term in the sum, when $\sigma=\lambda$, produces the coefficient of $v$ in $[P(\lambda):L(\mu)](v)$; and $\left([\Del(\lambda):L(\lambda)](v)\right)\left([\Del(\lambda):L(\mu)](v)\right)=[\Del(\lambda):L(\mu)](v)$. Therefore the coefficient of $v$ in $[\Del(\lambda):L(\mu)](v)$ is equal to $\dim\Ext^1(L(\lambda),L(\mu))$. 

Now, for the first direction of the proof, assume that $B$ has tight multiplicities. Let $\mu<\lambda$. Then $$\dim\Hom(\Del(\mu),\Del(\lambda))=[\Del(\lambda):L(\mu)]=[\Del(\lambda):L(\mu)](v)|_{v=1}$$ 
and the latter is $\dim\Ext^1(L(\lambda),L(\mu))$ plus the sum of the coefficients of $v^k$, $k>1$. On the other hand, we have $$\dim\Hom(\Del(\mu),\Del(\lambda))=\dim\Hom_{\mathrm{prim}}(\Del(\mu),\Del(\lambda))+\dim\Hom_{\mathrm{imprim}}(\Del(\mu),\Del(\lambda))$$ If $\phi:\Del(\mu)\rightarrow\Del(\lambda)$ factors through a highest weight module $M(\nu)$, $\mu<\nu<\lambda$, then $\im\phi$ is a submodule of a submodule in $\Del(\lambda)$, so belongs to the radical of the radical of $\Del(\lambda)$; thus it doesn't give rise to an extension between $L(\lambda)$ and $L(\mu)$ and so it doesn't contribute to the coefficient of $v$ in $[\Del(\lambda):L(\mu)](v)$. Any primitive hom $\phi:\Del(\mu)\rightarrow\Del(\lambda)$ gives rise to an extension $0\rightarrow L(\mu)\rightarrow E\rightarrow L(\lambda)\rightarrow 0$ since the head of $\im(\phi)$ must belong to $\rad_1\Del(\lambda)\subset\rad_1 P(\lambda)$ if $\phi$ does not factor through any other standard module. So $\dim\Hom_{\mathrm{prim}}(\Del(\mu),\Del(\lambda))\leq\dim\Ext^1(L(\lambda),L(\mu))$ in general, with equality in the situation of tight multiplicities.

For the converse direction of the proof, assume $Q_{\Ext^1}(B)$ is the double of $Q_{\mathrm{prim}}(B)$. Take $\Del(\lambda)\in B$. We will show that $\Rad\Del(\lambda)$ is generated by singular vectors. Let $\mu_1,...,\mu_k$ be a minimal set of $W$-irreps generating $\Rad\Del(\lambda)$. Suppose that for some $\mu_i$ there is no map $\Del(\mu_i)\rightarrow\Del(\lambda)$. Since $\mu_i$ is part of a minimal generating set for $\Rad\Del(\lambda)$, $H_c\cdot\mu_i$ does not belong to $\Rad(\Rad\Del(\lambda))$. Note that $\rad_1\Del(\lambda)=\Rad\Del(\lambda)/\Rad(\Rad\Del(\lambda))$ is semisimple, so the head of $H_c\cdot\mu_i$, which is $L(\mu_i)$, belongs to $\rad_1 \Del(\lambda)\subset\rad_1 P(\lambda)$. But this implies that $\dim\Ext^1(L(\lambda),L(\mu_i))>\dim\Hom_{\mathrm{prim}}(\Del(\mu_i),\Del(\lambda))$, contradicting the assumption.
\end{proof}

Now Lemma \ref{gr bgg} and Theorem \ref{quiver theorem} imply:
\begin{corollary}\label{quiver criterion}
The principal block $B_0$ of $\OO_c(G(r,1,n))$ at $c_0=\ell/n$, $\mathrm{gcd}(\ell,n)=1$, has tight multiplicities if and only if the quivers $Q_{\Ext^1}(B_0)$ and $Q_{\mathrm{prim}}(B_0)$ are the same.
\end{corollary}

\subsection{Graded decomposition numbers}

Let $B_0$ be the principal block of $\OO_c(G(r,1,n))$ in the case that $B_0$ is $\ttt$-diagonalizable and has tight multiplicities. Define a graph $\Gamma$ for $B_0$ as follows. The vertices of $\Gamma$ are all $\lambda\in B_0$. There is an arrow $\mu\rightarrow\lambda$ if and only if $\mu$ is the lowest degree subspace of a fundamental submodule of $\Del(\lambda)$. The graph $\Gamma$ encodes much of the structure of the block. 

Let $P_\lambda$ be the subgraph of $\Gamma$ coinciding with the lattice of intersections of the fundamental submodules of $\Del(\lambda)$ defined in Theorem \ref{submodule structure}.  Let $d(\mu,\lambda)$ be the length of the longest chain of arrows from $\mu$ to $\lambda$ in $\Gamma$, if such a chain exists.

\begin{theorem}\label{graded dec}
The graded decomposition numbers are given by
$$[\Delta(\lambda):L(\mu)](v)=v^{d(\mu,\lambda)}\hbox{ if $\mu$ is a vertex on }P_\lambda$$
Otherwise, $[\Delta(\lambda):L(\mu)](v)=0$.
\end{theorem}
\begin{proof}
We know that $[\Del(\lambda):L(\mu)]=\dim\Hom(\Del(\mu),\Del(\lambda))$ is $0$ if $\mu$ is not the highest weight of the intersection of fundamental submodules of $\Del(\lambda)$, and $1$ if it is, because any nonzero homomorphism $\Del(\mu)\rightarrow\Del(\lambda)$ is the inclusion of an intersection of fundamental submodules by the remarks in the proof of Theorem \ref{tight thm}, and $B_0$ has tight multiplicities. Also, we know by Lemma \ref{Jantzen lemma} that the power of $v$ in the graded decomposition number is equal to the number of fundamental submodules of $\Del(\lambda)$ containing the homomorphic image of $\Del(\mu)$. If the longest length of a chain of arrows $\mu\rightarrow\lambda$ is $1$ then $\mu$ labels a fundamental submodule of $\Del(\lambda)$ which is not contained in any other fundamental submodule. Suppose by induction that all the vertices $\nu$ corresponding to intersections of $k$ fundamental submodules have distance $k$ from $\lambda$, for any $k\leq d$, where $d$ is less than the number of fundamental submodules of $\Del(\lambda)$. Suppose $\mu$ is the highest weight of the intersection of exactly $d+1$ fundamental submodules. Then the image of $\Del(\mu)$ is a proper submodule of the intersection $I$ of some $d$ fundamental submodules. Say the highest weight of $I$ is $\nu$. Then the longest path $\nu\rightarrow\lambda$ has length $d$, so there is a path of length at least $d+1$ from $\mu$ to $\lambda$. Suppose there is a path of length greater than $d+1$, whose first arrow is $\mu\rightarrow\xi$. Then $\xi$ must be the intersection of $\ell\geq d+1$ fundamental submodules since any intersection of a smaller collection of fundamental submodules satisfies the induction hypothesis. But then the submodule with highest weight $\mu$ is a proper submodule of the submodule with highest weight $\xi$, so it must be the intersection of more than $d+1$ fundamental submodules. Therefore $d+1$ is the longest length of a path $\mu\rightarrow\lambda$, which completes the induction. 
\end{proof}

The $|\Lambda|\times|\Lambda|$ matrix with entries $[\Del(\lambda):L(\mu)](v)$ is the \textit{graded decomposition matrix}. If rows and columns are arranged so that the partial order $>$ on $\Lambda$ is nonincreasing across rows and down columns then the matrix is upper-triangular with $1$'s on the diagonal.

\subsection{The inverse of the graded decomposition matrix}
We now describe a formula for the inverse of the graded decomposition matrix whose entries were given by Theorem \ref{graded dec}. This formula should really hold in any tight block with simple multiplicities equipped with a minimal partial order. Of course, such a formula is given by abstract Kazhdan-Lusztig theory in terms of the Poincar\'{e} polynomial $p_{\mu,\lambda}$ evaluated at $-1$, c.f. Proposition (3.2) in \cite{CPS2}, if we can calculate $\dim\Hom^n(L(\lambda),\nabla(\mu))$. In a tight block, however, everything is controlled by homomorphisms between standards, and we may define a naive polynomial which does the same job as the Poincar\'{e} polynomial when evaluated at $-1$.
 
Let $\Gamma(\Hom)$ be the graph with vertices $\lambda\in\Lambda$ and an arrow $\mu\rightarrow\lambda$ if there is a nonzero homomorphism $\Del(\mu)\rightarrow\Del(\lambda)$. For $\mu\leq\lambda$, let $b_n(\mu,\lambda)$ be the number of chains of $n$ arrows from $\mu$ to $\lambda$ in $\Gamma(\Hom)$.
Now we define a polynomial in $\NN[v]$ which keeps track of all chains of arrows from $\mu$ to $\lambda$ in $\Gamma(\Hom)$:
$$B_{\mu,\lambda}=\sum_{n=0}^{d(\mu,\lambda)} b_n(\mu,\lambda) v^n$$

\begin{remark} By Theorem \ref{graded dec}, the number of chains of homs of length $n$ from $\mu$ to $\lambda$ is the number of paths in the $v$-dec matrix starting at $(\mu,\mu)$, ending at $(\lambda,\lambda)$, traveling up and left and bouncing off exactly $n$ nonzero $v$-dec numbers and $n-1$ times off the diagonal. The product of the entries on the bounce points of such a path is $v^{d(\mu,\lambda)}$.
\end{remark}

\begin{lemma}\label{inverse graded dec}
$$[L(\lambda)](v)=\sum_{\mu\leq \lambda} B_{\mu,\lambda}(-1)v^{d(\mu,\lambda)}[\Del(\mu)]$$
\end{lemma}
\begin{proof}
Take $\lambda\in\Lambda$. To find $[L(\lambda):\Del(\mu)](v)$, i.e. the coefficient of $[\Del(\mu)]$ in the formula, we will induct on $d(\mu,\lambda)$. For $\mu=\lambda$, $B_{\lambda,\lambda}=1$ and $d(\lambda,\lambda)=0$, and we have $[L(\lambda):\Del(\lambda)](v)=1$ which is true. Assume by induction that if $d(\sigma,\lambda)<n$ that $[L(\lambda):\Del(\sigma)](v)=B_{\sigma,\lambda}(-1)v^{d(\sigma,\lambda)}$. Take $\mu$ such that $d(\mu,\lambda)=n$. 

Consider the first arrow in a chain from $\mu$ to $\lambda$ in $\Gamma(\Hom)$, the arrow originating at $\mu$. If $\sigma$ is the target of an arrow originating at $\mu$, then the contribution to $B_{\mu,\lambda}$ from all those chains $\mu$ to $\lambda$ in $\Gamma(\Hom)$ which pass through $\sigma$ is: $vB_{\sigma,\lambda}$. Therefore: $$B_{\mu,\lambda}=v\left(\sum_{\substack{\mu<\sigma\leq\lambda \\ \dim\Hom(\Del(\mu),\Del(\sigma))\neq0}}B_{\sigma,\lambda}\right)$$
But if $\dim\Hom(\Del(\mu),\Del(\sigma))\neq0$, then $[\Del(\sigma):L(\mu)](v)=v^{d(\mu,\sigma)}$ by Theorem \ref{graded dec}, and if $\sigma>\mu$ then by definition $d(\mu,\sigma)\geq1$. Then $d(\sigma,\lambda)<n$ since $d(\mu,\lambda)=d(\mu,\sigma)+d(\sigma,\lambda)$. Applying induction, $[L(\lambda):\Del(\sigma)](v)=B_{\sigma,\lambda}(-1)v^{d(\sigma,\lambda)}$. Now observe that the dot product of the row vector $[L(\lambda):\Del(-)](v)$ with the column vector $[\Del(-):L(\mu)](v)$ is $0$, since the $|\Lambda|\times|\Lambda|$ matrix with entries $[L(\sigma):\Del(\tau)](v)$ is the inverse of the matrix with entries $[\Del(\sigma):L(\tau)](v)$. Then we have: $$[L(\lambda):\Del(\mu)](v)=-\left(\sum_{\substack{\mu<\sigma\leq\lambda\\ \dim\Hom(\Del(\mu),\Del(\sigma)\neq0}} B_{\sigma,\lambda}(-1)v^{d(\sigma,\lambda)}v^{d(\mu,\sigma)}\right)=B_{\mu,\sigma}(-1)v^{d(\mu,\lambda)}$$
\end{proof}

\subsection{BGG resolutions}\label{bgg res}

We describe an algorithm which computes the standards appearing in a minimal BGG resolution of any $L(\lambda)$ in $B_0$. Let $\Lambda_{\leq\lambda}=\{\mu\leq\lambda\}$ be the poset ideal of $\Lambda$ generated by $\lambda$. Define the subgraph $\Gamma_\lambda$ of $\Gamma$ to be that containing all $\mu\leq\lambda$ and their arrows.
Set 
$$\Lambda_\lambda^i=\{\mu\in\Lambda_{\leq\lambda}\;|\;d(\mu,\lambda)=i\}$$
Next, any time there's an arrow $\nu\rightarrow\mu$ with $\nu\in\Lambda_\lambda^i$ and $\mu\in\Lambda_\lambda^j$ and $i-j>1$, delete that arrow and remove $\Gamma_\nu$ from $\Gamma_\lambda$. The graph $\bar{\Gamma}_\lambda$ which remains is, essentially, the complex of standard modules resolving $L(\lambda)$: conjecturally, the standards in the minimal BGG resolution $\Del_\bullet\rightarrow L(\lambda)\rightarrow 0$ are $$\Del_i=\bigoplus_{\tau\in\bar{\Gamma}_\lambda\cap\Lambda_\lambda^i}\Del(\tau).$$

\begin{conjecture}\label{bgg conj}This algorithm constructs the inverse of the $v$-decomposition matrix of the principal block $B_0$ described in Theorem \ref{graded dec}: $[L(\lambda):\Del(\mu)](v)=(-1)^{d(\mu,\lambda)}v^{d(\mu,\lambda)}$ if $\mu\in\bar{\Gamma}_\lambda$, and $0$ otherwise.
\end{conjecture}

Row $\lambda$ of the inverse of the graded decomposition matrix is given by the formula \ref{inverse graded dec}. Thus Conjecture \ref{bgg conj} is equivalent to the claim that if $\mu\leq\lambda$, then, 
\begin{enumerate}
\item \label{(a)} $B_{\mu,\lambda}(-1)=(-1)^{d(\mu,\lambda)}$ if every path from $\mu$ to $\lambda$ in $\Gamma$ has the same length $d(\mu,\lambda)$
\item $B_{\mu,\lambda}(-1)=0$ if there are paths from $\mu$ to $\lambda$ of different lengths in $\Gamma$.
\end{enumerate}
Let us prove part (\ref{(a)}) by induction:
\begin{proof}(\ref{(a)}) Suppose every path $\mu$ to $\lambda$ in $\Gamma$ has length $d(\mu,\lambda)$. Then all vertices $\nu$ lying on a path from $\mu$ to $\lambda$ in $\Gamma$ also have this property. Consider all nonzero homs out of $\Del(\mu)$ which lie on a path of homs to $\lambda$: the subgraph $C$ of these lying in $\Gamma_\lambda$ is the edge graph of an $N$-cube, where $N$ is the number of outgoing arrows from $\mu$ in $\Gamma_\lambda$. There is a hom from $\Del(\mu)$ to $\Del(\sigma)$ for each vertex $\sigma$ of $C$ by \ref{submodule structure}: $C$ is the lattice of intersections of some collection of fundamental submodules of a standard module which have pairwise intersections not equal to any fundamental submodule, plus the apex of $C$ which is the standard module itself. Then we have:
$$B_{\mu,\lambda}=v\sum_{j=1}^N\sum_{\substack{\sigma\in C \\ d(\mu,\sigma)=j}}p_{\sigma,\lambda}$$
Evaluating at $v=-1$ and using induction gives:
$$B_{\mu,\lambda}(-1)=(-1)\sum_{j=1}^N{N\choose j}(-1)^{d(\mu,\lambda)-j}=(-1)\left(\sum_{j=0}^N{N\choose j}(-1)^{d(\mu,\lambda)-j}-(-1)^{d(\mu,\lambda)}\right)=(-1)^{d(\mu,\lambda)}$$
\end{proof}

\section{Examples}\label{examples}
We illustrate the material of the preceding sections with pictures of diagonalizable principal blocks with tight multiplicities and of BGG resolutions of finite-dimensional modules. The graphs $\Gamma$ below are built up inductively, starting by putting vertices for those $\lambda$ in the principal block $B_0$ such that there cannot be any $\nu\neq\lambda$ in $B_0$ with a nonzero map $\Del(\lambda)\rightarrow\Del(\nu)$. Next, for a given vertex $\lambda$ in $\Gamma$, there is a vertex for each submodule of the form  $M=M_{b,?}$ or $M_{b_1,b_2,?}$ of $\Del(\lambda)$, together with an arrow $\mu\rightarrow\lambda$. Each such vertex for a submodule $M\subset \Del(\lambda)$ is labeled by the irreducible representation $\mu$ of $\CC G(r,1,rn)$ that occupies the lowest degree subspace of $M$. $M\subset\Del(\lambda)$ then determines a map $\Del(\mu)\rightarrow\Del(\lambda)$, since the lowest degree subspace always consists of singular vectors.

 The arrows in shades of red and orange denote the primitive homs between standards, while the green arrows are nonzero compositions of homs. To obtain the $\Ext^1$ quiver for simples, delete the green arrows and double the orange arrows. 

\subsection{BGG resolutions and character formulas} 

According to \ref{bgg conj}, to find the standards in a minimal BGG resolution of some $L(\lambda)$, one should take the subgraph $\Gamma_\lambda$ consisting of partitions with paths to $\lambda$, and delete from it anything that lies on a path through a green arrow to $\lambda$. The result is the graph called $\bar{\Gamma}_\lambda$ which conjecturally ``is" the BGG resolution. We have included several examples of $\bar{\Gamma}_\chi$ for finite-dimensional modules $L(\chi)$, $\chi$ a twist of the trivial rep. $\Del_i$, the $i$'th term in the resolution $\Del_\bullet\rightarrow L(\Triv)\rightarrow 0$, is the direct sum of the $\Del(\lambda)$ where $\lambda$ occurs on the $i$'th level from the top of $\bar{\Gamma}_\chi$ (the character $\chi$ is at level $0$). We find that for $G(3,1,3)$, $G(3,1,6)$, and $G(4,1,4)$ at equal parameters $1/rn$, this produces the right dimensions of finite-dimensional reps according to Corollary \ref{rdim}. Furthermore, we have checked that Conjecture \ref{bgg conj} actually holds for the whole principal block in the cases of $G(3,1,3)$ at $c=1/3$, and $G(4,1,4)$ at $c=1/4$, by computing the $v$-decomposition matrix using \ref{graded dec} and finding its inverse and checking this matches the answer given by our graphical algorithm.

We may then calculate the graded dimension formula for any $L(\lambda)$ in a tight, diagonalizable principal block of $\OO_c$ using the BGG resolution of $L(\lambda)$ as follows: for $\mu$ in the graph $\bar{\Gamma}_\lambda$ , let $\tilde{c}(b)=rc_0\ct(b)+d_{\beta(b)}$ for each box $b$ in $\mu$ and let 
\begin{equation}\label{charge}
\tilde{c}_\lambda(\mu)=\left(\sum_{b\in\lambda}\tilde{c}(b)\right)-\sum_{b\in\mu}\tilde{c}(b)
\end{equation}
Note that $\tilde{c}_\lambda$ is the same function as that given by the Euler element $h_c$, except shifted so that $\Del(\lambda)$ is $\ZZ_{\geq0}$-graded with $\lambda$ in degree $0$. When $\lambda=\Triv$ we will write $\tilde{c}$ for $\tilde{c}_\Triv$. Let $\hd(\mu)=i$ be the homological degree of $\mu$, i.e. the unique $i$ such that $\Del(\mu)$ is a direct summand of $\Del_i$.  Thus $\hd(\mu)=i$ if $\mu$ appears on the $i$'th  level from the top of the graph $\bar{\Gamma}_\lambda$, starting from level $0$ for $\lambda$, i.e $\hd(\mu)=d(\mu,\lambda)$. The \textit{graded dimension} of $L(\lambda)$ is given by the formula:
\begin{equation}\label{gr char}
\sum_{t\geq 0}\dim_\CC (L(\lambda)^i) t^i=\frac{\sum_{\mu\in\bar{\Gamma}_\lambda}(-1)^{\hd(\mu)}\dim_\CC(\mu) t^{\tilde{c}_\lambda(\mu)}}{(1-t)^n}
\end{equation}

\subsection{The principal block of $\OO_{1/rn}(G(r,1,rn))$} In the examples below, we write $c=1/rn$ as shorthand for $c_0=c_1=...=c_{r-1}=1/rn$. We have $d_0=(r-1)/rn$ and $d_i=-1/rn$ if $1\leq i \leq r-1$. Let $B_0$ be the principal block of $\OO_c(G(r,1,rn))$. Every standard module in $B_0$ is $\ttt$-diagonalizable by Lemma \ref{diag}. Following Section \ref{indexation} we have a labeling of $\Triv=(rn)$ with a distinct box $b_i$ for each $0\leq i\leq r-1$. The box $b_0$ is the box with content $0$, so it is the first box, and $(0,k_0)=(0,0)$. For $0<i\leq r-1$, $b_i$ is the $(r-i)n$'th box, and $(i,k_i)=(i,-1)$. Read from left to right, we have 
$$(0,0)>(r-1,-1)>(r-2,-1)>...>(2,-1)>(1,-1)>(0,-1)$$
so by Theorem \ref{tight thm}, the principal block $B_0$ has tight multiplicities.

The partitions in $B_0$ may be found as follows (see \ref{submodules}). Label the Young diagram of $(rn)$ with the numbers $1$ through $rn$ from left to right, then connect up the two ends so that the numbers $1$ through $rn$ are arranged increasing clockwise around a circle. Make $k$ cuts to make $k$ connected subsets $Y_1,...,Y_k$, $k \leq r$, such that each $Y_a$ for $a>1$ contains a multiple of $n$ (if $Y_a$, $a>1$, does not contain a box labeled by a multiple of $n$ then the collection $Y_1,...,Y_k$ cannot give rise to a collection of partitions in the block). For each $1\leq b \leq r-1$, $bn\in Y_a$ for some $a$. Then $Y_a$ may be turned into a (labeled) hook $\lambda^{r-b}$ and put in component $r-b$, such that $\lambda^{r-b}$ has $bn$ in the upper left corner and all numbers increasing (mod $rn$) from bottom to top, left to right. Additionally, $Y_1$ may be placed in the $0$'th component if it is bent into the hook $\lambda^0$ having $1$ in the upper left corner, with numbers increasing left to right across the arm, and bottom to top up the leg, except that $rn$ may occur in the box below $1$. $\lambda$ is the $r$-partition of $rn$ produced by bending some such collection of $Y_a$ at multiples of $n$ or at $1$ and placing them in the appropriate components.

In the case $r=2$, i.e. for $B_{2n}$, this amounts to the following characterization:
\begin{lemma} The bipartitions $(\lambda,\mu)$ in the principal block of $\OO_{1/2n}(B_{2n})$ consist of:
\begin{itemize}
\item$(\lambda,\emptyset)$, $\lambda\vdash 2n$ a hook
\item$(\emptyset,\mu)$, $\mu\vdash 2n$ a hook
\item$((a,1^k),(n+1-k,1^{n-1-a}))$, $1\leq a\leq n-1$, $0\leq k\leq n$. 
\end{itemize}
\end{lemma}
The principal block for $B_{2n}$ at equal parameters $c=1/2n$ has weight $2$ by \cite{Fa}, Proposition 1.10. Weight $2$ blocks for Iwahori-Hecke algebras of type $B_n$ were studied by Fayers in \cite{Fa}. The image of $B_0$ under $\KZ$ functor coincides with the ``prototype" of a Type I block in Fayers' classification (\cite{Fa}, Propositions 2.1, 2.2).

In the case of the principal block for $\OO_{1/2n}(B_{2n})$, we may explicitly describe all maps between standards in the block in terms of bipartitions. The calculation and case by case description is lengthy; we summarize and refer the reader to the examples for $B_{2n}$, $n=1,2,3,4$ below (\ref{B(2)B(4)B(6)},\ref{B(8)}). The graph $\Gamma$ admits a planar embedding such that given a vertex $\lambda\in\Gamma$, the subgraph $P_\lambda$ spanned by those $\mu$ such that there is a nonzero map $\Del(\mu)\rightarrow\Del(\lambda)$ is either a point, a directed line segment connecting two points or, in the case of $\lambda=((1^{n+1},1^{n-1})$, two directed line segments to $\lambda$, or, encloses a two-dimensional polygon. The polygons which occur in $\Gamma$ are either triangles or diamonds. In fact, there are only two triangles which arise in $\Gamma$, and they are: 
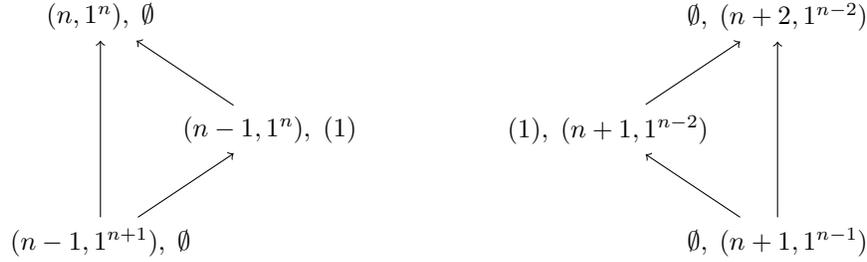
\begin{center}
\begin{tikzpicture}[scale=1.5,font=\footnotesize,every text node part/.style={align=center}]
\node(top1) at (-3,0) {$(n,1^n),\;\emp$};
\node(mid1) at (-1.5,-1) {$(n-1,1^n),\;(1)$};
\node(bot1) at (-3,-2) {$(n-1,1^{n+1}),\;\emp$};
\node(top2) at (3,0) {$\emp,\;(n+2,1^{n-2})$};
\node(mid2) at (1.5,-1) {$(1),\;(n+1,1^{n-2})$};
\node(bot2) at (3,-2) {$\emp,\;(n+1,1^{n-1})$};
\path[->]
(mid1) edge node {}(top1)
(bot1) edge node {}(mid1)
(bot1) edge node {}(top1)
(mid2) edge node {}(top2)
(bot2) edge node {}(mid2)
(bot2) edge node {}(top2);
\end{tikzpicture}
\end{center}
To read the graded decomposition numbers off of $\Gamma$, look at the vertex $\lambda$ in the graph, and look at the polygon whose edges flow into vertex $\lambda$: then $[\Del(\lambda):L(\mu)](v)=0$ unless $\mu$ is a vertex in that polygon, in which case $[\Del(\lambda):L(\mu)](v)=1$ if $\lambda=\mu$, $v$ if there is a single path $\mu\rightarrow\lambda$, and $v^2$ if the path $\mu\rightarrow\lambda$ traverses one side of a single triangle or diamond. 

Theorem \ref{intro bgg thm} now follows from the discussion above and \ref{inverse graded dec}.
\begin{corollary}\label{bgg B2n} Let $\Del_i$ denote the coefficient of $v^i$, up to signs, in the graded character of $L_{1/2n}(\Triv)$ in $\OO_{1/2n}(B_{2n})$. The $\mu$ appearing as highest weights of summands of $\Del_i$ are those obtained from $((2n),\emp)$ by $i$ single-chunk-of-boxes moves of the forms described in Corollary \ref{lowest degree subspace}, and which can't be so obtained by a smaller number of such moves. This proves Theorem \ref{intro bgg thm}. Thus every $\mu\in\Lambda$ appears in the character except for those $\mu=(\mu^0,\mu^1)$ such that either $\mu^0=\emp$ and $\mu^1$ has more than $n+1$ boxes in the first row, or $\mu^1=\emp$ and $\mu^0$ has more than $n+1$ boxes in the first column. 
\end{corollary}
\begin{proof}
If $\mu=(\emp,\mu^1)$ and $\mu^1$ has more than $n+1$ boxes in the first row, then $\mu^1$ is not comparable to $\Triv$ in the partial order, so does not appear in the character formula. If $\mu=(\mu^0,\emp)$ and $\mu^0$ has more than $n+1$ boxes in the first column, then we may calculate $B_{\mu,\Triv}$, and thus $[L(\lambda)](v)$ by \ref{inverse graded dec}, as follows. Let $\mu^0=(b,1^{2n-b})$ with $a:=2n-b\geq n+1$. There is a single chain of homs $\mu\rightarrow\Triv$ which passes through the green arrow and this path has length $a$. Any other chain of homs $\mu\rightarrow\Triv$ passes through one of the vertices $((n-1,1^n),(1))$, $((n-2,1^n),(1^2))$,...,$((b,1^n),(1^{n-b}))$. Given one of these vertices $\nu$, any path from $\nu$ to $\Triv$ in $\Gamma$ has length $d(\nu,\Triv)$. For such a $\nu$, there is a single path from $\mu$ to $\nu$ along arrows from $\Gamma$ which does not go through the other vertices in this list lying below $\nu$ in the partial order. Then there is a single path of homs from $\mu$ to $\nu$, unless $\nu$ is $((b,1^n),(1^{n-b}))$, that goes across the diagonal of a diamond with top vertex $\nu$. We then have 
$$B_{\mu,\Triv}=v^a+vB_{((b,1^n),(1^{n-b})),\Triv}+\sum_{k=1}^{n-b-1} (v^k+v^{k+1})B_{((b+k,1^n),(1^{n-b-k})),\Triv}$$
and so
$$
B_{\mu,\Triv}(-1)=(-1)^a+(-1)^{d(((b,1^n),(1^{n-b})),\Triv)+1}=(-1)^a+(-1)^{a+1}=0
$$
Then for these $\mu$, $[L(\Triv):\Del(\mu)](v)=0$ by \ref{inverse graded dec}.

Next, by the part of Conjecture \ref{bgg conj} which we proved (\ref{(a)}), if every path $\mu$ to $\Triv$ in $\Gamma$ has length $d(\mu,\Triv)$ then $[L(\Triv):\Del(\mu)]=(-1)^{d(\mu,\Triv)}v^{d(\mu,\Triv)}$. This implies the formula.
\end{proof}
The same argument works for any $\lambda$ in place of $\Triv$ with $\lambda\geq((n,1^n),\emp)$. A similar argument works for $\lambda\geq(\emp,(n+2,1^{n-2}))$. Any $\lambda\in B_0$ is comparable to at most one of $((n,1^n),\emp)$, $(\emp,(n+2,1^{n-2}))$. Thus Conjecture \ref{bgg conj} holds for $L(\lambda)$ in either case. If $\lambda$ is not comparable to either of $((n,1^n),\emp)$, $(\emp,(n+2,1^{n-2}))$, then any $\mu\leq\lambda$ satisfies the conditions of (\ref{(a)}) and so the character formula of Conjecture \ref{bgg conj} holds by the proof of (\ref{(a)}). This verifies Conjecture \ref{bgg conj} for the principal block of $B_{2n}$ at equal parameters $c=1/2n$, the case concerning the Oblomkov-Yun dimension formula proved in \ref{dim formula}.

\subsection{Graded dimensions}
The pictures below contain all information necessary to calculate graded dimensions of simple modules. In $\Gamma$, vertical spacing between partitions is proportionate to difference between their charged contents, and the minimum of such in $\Gamma$ is $1$. In  BGG resolutions, latitudes are homological degrees. We then obtain graded dimensions and dimensions over $\CC$ of the finite-dimensional representations, which agree, in the case of $L(\Triv)$, with Corollary \ref{rdim}:

In \ref{G(3)ch}, let $\chi_1=(\emp,(3),\emp)$ and $\chi_2=(\emp,\emp,(3))$. We have $\tilde{c}_{\chi_i}(\lambda)=2-\sum_{b\in\lambda}\tilde{c}(b)$, see \ref{charge}, while $\tilde{c}(\lambda)=5-\sum_{b\in\lambda}\tilde{c}(b)$. Applying formula \ref{gr char}:
\begin{align*}
\sum_{t\geq 0}\dim_\CC(L(\Triv)^i) t^i &= 1+3t+6t^2+7t^3+3t^4\\
\sum_{t\geq 0}\dim_\CC (L(\chi_1)^i) t^i
&=1\\
\sum_{t\geq 0}\dim_\CC (L(\chi_2)^i) t^i
&=1+3t+3t^2+t^3
\end{align*}
Evaluating at $t=1$, we get $\dim L(\Triv)=20$, $\dim L(\emp,(3),\emp)=1$, and $\dim L(\emp,\emp,(3))=8$.

In \ref{G(3,6)triv},\ref{G(3,6)ch,G(4)} let $\Triv=((6),\emp,\emp)$, $\chi_1=(\emp,(6),\emp)$, and $\chi_2=(\emp,\emp,(6))$. Then $\tilde{c}(\lambda)=9.5-\sum_{b\in\lambda}\tilde{c}(b)$, and $\tilde{c}_{\chi_i}(\lambda)=6.5-\sum_{b\in\lambda}\tilde{c}(b)$, see \ref{charge}. Applying formula \ref{gr char}:
\begin{align*}
\sum_{t\geq 0}\dim_\CC (L(\Triv)^i) t^i
&=1+6t+21t^2+51t^3+96t^4+141t^5+141t^6+66t^7+15t^8\\
\sum_{t\geq 0}\dim_\CC (L(\chi_1)^i) t^i
&=1+6t+21t^2+31t^3+6t^4\\
\sum_{t\geq 0}\dim_\CC (L(\chi_2)^i) t^i
&=1+6t+15t^2+20t^3
\end{align*}
Evaluating at $t=1$, we get $\dim L(\Triv)=538$, $\dim L(\emp,(6),\emp)=65$, and $\dim L(\emp,\emp,(6))=42$.

In \ref{G(4)triv}, we have $\tilde{c}(\lambda)=9-\sum_{b\in\lambda}\tilde{c}(b)$, while $\tilde{c}_\chi(\lambda)=5-sum_{b\in\lambda}\tilde{c}(b)$ for $\chi\neq\Triv$ a character with $(4)$ in a single component, see \ref{charge}. Applying formula \ref{gr char}: 
\begin{align*}
\sum_{t\geq 0}\dim_\CC (L(\Triv)^i) t^i
&=1+4t+10t^2+20t^3+31t^4+40t^5+34t^6+12t^7\\
\sum_{t\geq 0}\dim_\CC (L(\emp,(4),\emp,\emp)^i) t^i&=1+4t+4t^2\\
\sum_{t\geq 0}\dim_\CC (L(\emp,\emp,(4),\emp)^i) t^i&=1+4t+6t^2\\
\sum_{t\geq 0}\dim_\CC (L(\emp,\emp,\emp,(4))^i) t^i&=1+4t+10t^2+16t^3+13t^4+4t^5\\
\end{align*}
Evaluating at $t=1$, we get $\dim L(\Triv)=152$, $\dim L(\emp,(4),\emp,\emp)=9$, $\dim L(\emp,\emp,(4),\emp))=11$, and $\dim L(\emp,\emp,\emp,(4))=48$.

\subsection{The principal blocks of $B_2$ at $c=1/2$, $B_4$ at $c=1/4$, and $B_6$ at $c=1/6$.}
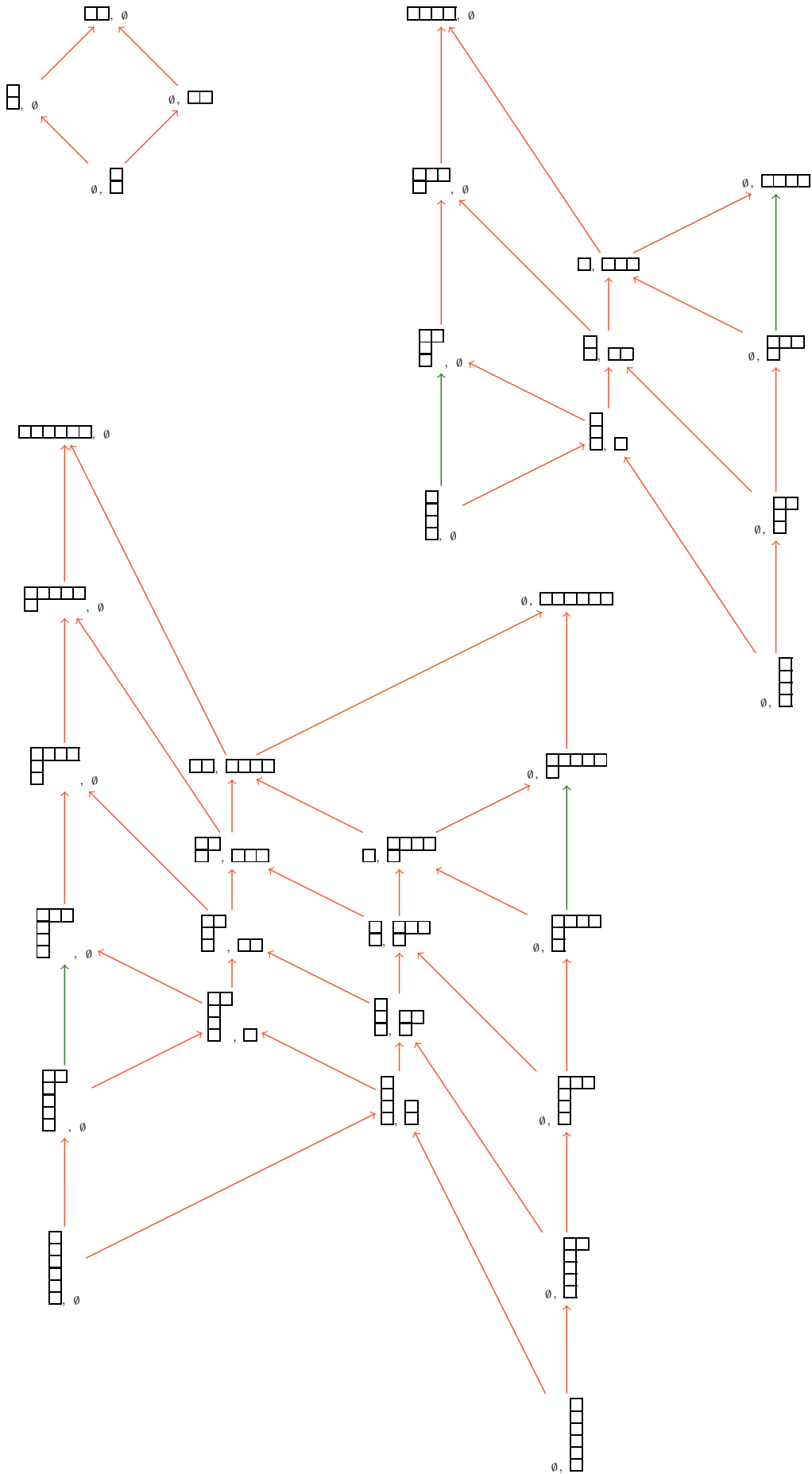
\label{B(2)B(4)B(6)}
\begin{center}
{\fontsize{5}{6}\selectfont
\begin{tikzpicture}[scale=1.4,every text node part/.style={align=center}]
\node (E11) at (-3,4) {$\yng(2),\;\emp$};
\node (E21) at (-2,3) {$\emp,\;\yng(2)$};
\node (E22) at (-4,3) {$\yng(1,1),\;\emp$};
\node (E31) at (-3,2) {$\emp,\;\yng(1,1)$};

\node(f11) at (1,4) {$\yng(4),\;\emp$};
\node(f21) at (1,2) {$\yng(3,1),\;\emp$};
\node(f22) at (5,2) {$\emp,\;\yng(4)$};
\node(f31) at (3,1) {$\yng(1),\;\yng(3)$};
\node(f41) at (1,0) {$\yng(2,1,1),\;\emp$};
\node(f42) at (3,0) {$\yng(1,1),\;\yng(2)$};
\node(f43) at (5,0) {$\emp,\;\yng(3,1)$};
\node(f51) at (3,-1) {$\yng(1,1,1),\;\yng(1)$};
\node(f61) at (1,-2) {$\yng(1,1,1,1),\;\emp$};
\node(f62) at (5,-2) {$\emp,\;\yng(2,1,1)$};
\node(f71) at (5,-4) {$\emp,\;\yng(1,1,1,1)$};

\node(g11) at (-3.5,-1) {$\yng(6),\;\emp$};
\node(g21) at (-3.5,-3) {$\yng(5,1),\;\emp$};
\node(g22) at (2.5,-3) {$\emp,\;\yng(6)$};
\node(g31) at (-3.5,-5) {$\yng(4,1,1),\;\emp$};
\node(g32) at (-1.5,-5) {$\yng(2),\;\yng(4)$};
\node(g33) at (2.5,-5) {$\emp,\;\yng(5,1)$};
\node(g41) at (-1.5,-6) {$\yng(2,1),\;\yng(3)$};
\node(g42) at (0.5,-6) {$\yng(1),\;\yng(4,1)$};
\node(g51) at (-3.5,-7) {$\yng(3,1,1,1),\;\emp$};
\node(g52) at (-1.5,-7) {$\yng(2,1,1),\;\yng(2)$};
\node(g53) at (0.5,-7) {$\yng(1,1),\;\yng(3,1)$};
\node(g54) at (2.5,-7) {$\emp,\;\yng(4,1,1)$};
\node(g61) at (-1.5,-8) {$\yng(2,1,1,1),\;\yng(1)$};
\node(g62) at (0.5,-8) {$\yng(1,1,1),\;\yng(2,1)$};
\node(g71) at (-3.5,-9) {$\yng(2,1,1,1,1),\;\emp$};
\node(g72) at (0.5,-9) {$\yng(1,1,1,1),\;\yng(1,1)$};
\node(g73) at (2.5,-9) {$\emp,\;\yng(3,1,1,1)$};
\node(g81) at (-3.5,-11) {$\yng(1,1,1,1,1,1),\;\emp$};
\node(g82) at (2.5,-11) {$\emp,\;\yng(2,1,1,1,1)$};
\node(g91) at (2.5,-13) {$\emp,\;\yng(1,1,1,1,1,1)$};

\path[->,semithick,RedOrange]
(E21) edge node [left] {} (E11)
(E22) edge node [right] {}(E11)
(E31) edge node [left] {} (E21)
(E31) edge node [right]{}(E22);

\path[->,semithick,RedOrange]
(f21) edge node {}(f11)
(f31) edge node {}(f11)
(f31) edge node {}(f22)
(f41) edge node {}(f21)
(f42) edge node {}(f21)
(f42) edge node {}(f31)
(f43) edge node {}(f31)
(f51) edge node {}(f41)
(f51) edge node {}(f42)
(f61) edge node {}(f51)
(f62) edge node {}(f42)
(f62) edge node {}(f43)
(f71) edge node {}(f51)
(f71) edge node {}(f62);
\path[->,semithick,OliveGreen]
(f43) edge node {}(f22)
(f61) edge node {}(f41);
\path[->,semithick,RedOrange]
(g21) edge node {}(g11)
(g31) edge node {}(g21)
(g32) edge node {}(g11)
(g32) edge node {}(g22)
(g33) edge node {}(g22)
(g41) edge node {}(g21)
(g41) edge node {}(g32)
(g42) edge node {}(g32)
(g42) edge node {}(g33)
(g51) edge node {}(g31)
(g52) edge node {}(g31)
(g52) edge node {}(g41)
(g53) edge node {}(g41)
(g53) edge node {}(g42)
(g54) edge node {}(g42)
(g61) edge node {}(g51)
(g61) edge node {}(g52)
(g62) edge node {}(g52)
(g62) edge node {}(g53)
(g71) edge node {}(g61)
(g72) edge node {}(g61)
(g72) edge node {}(g62)
(g73) edge node {}(g53)
(g73) edge node {}(g54)
(g81) edge node {}(g71)
(g81) edge node {}(g72)
(g82) edge node {}(g62)
(g82) edge node {}(g73)
(g91) edge node {}(g72)
(g91) edge node {}(g82);
\path[->,semithick,OliveGreen]
(g54) edge node {}(g33)
(g71) edge node {}(g51);
\end{tikzpicture}}
\end{center}

\subsection{The principal block of $\OO_{1/8}(B_8)$}
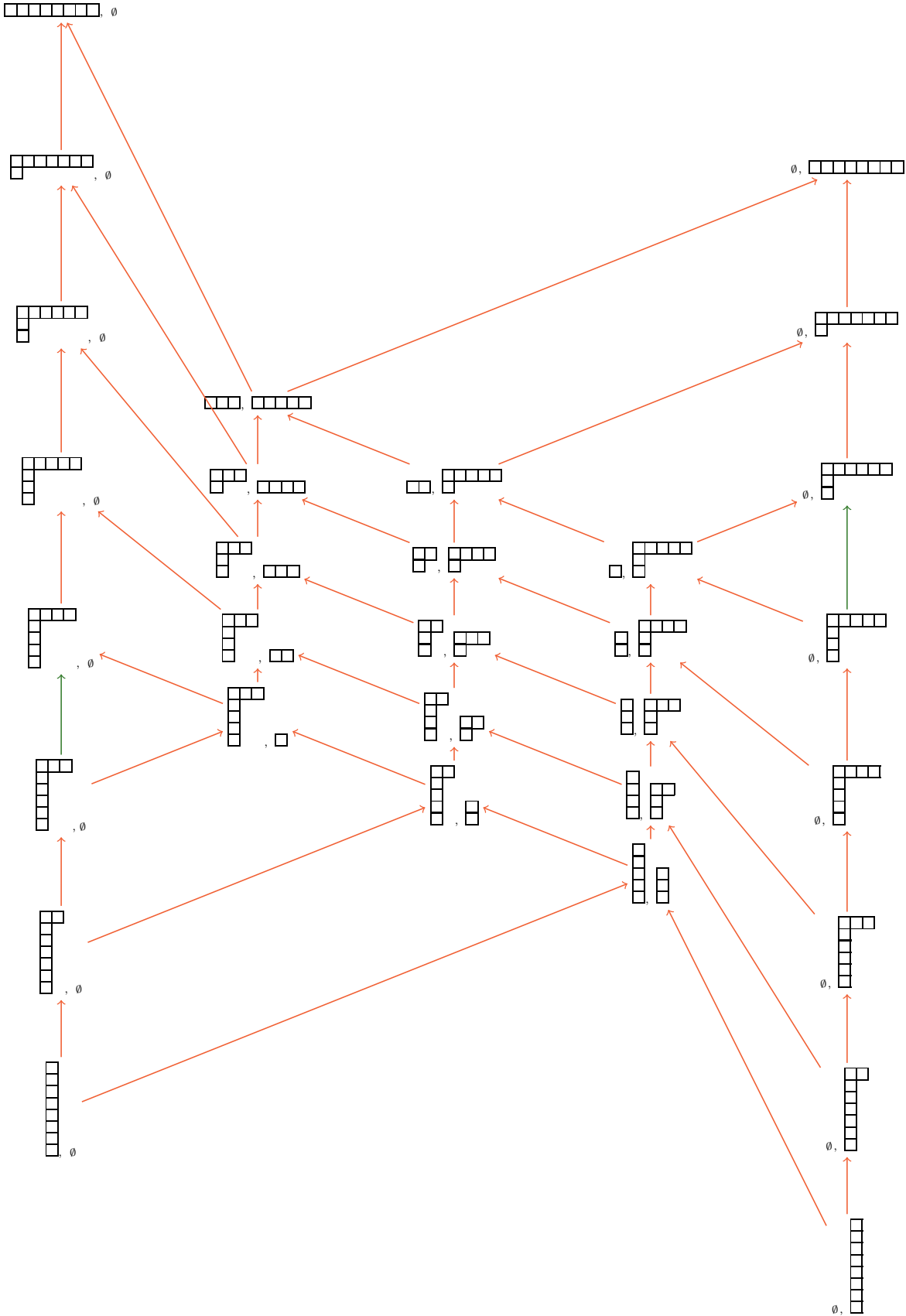
\label{B(8)}

\begin{center}
{\fontsize{5}{6}\selectfont
\begin{tikzpicture}[scale=1.35,every text node part/.style={align=center}]
\node(01) at (-5,0){$\yng(8),\;\emp$};
\node(11) at (-5,-2){$\yng(7,1),\;\emp$};
\node(12) at (5,-2){$\emp,\;\yng(8)$};
\node(21) at (-5,-4){$\yng(6,1,1),\;\emp$};
\node(22) at (5,-4){$\emp,\;\yng(7,1)$};
\node(31) at (-2.5,-5){$\yng(3),\;\yng(5)$};
\node(41) at (-5,-6){$\yng(5,1,1,1),\;\emp$};
\node(42) at (-2.5,-6){$\yng(3,1),\;\yng(4)$};
\node(43) at (0,-6){$\yng(2),\;\yng(5,1)$};
\node(44) at (5,-6){$\emp,\;\yng(6,1,1)$};
\node(51) at (-2.5,-7){$\yng(3,1,1),\;\yng(3)$};
\node(52) at (0,-7){$\yng(2,1),\;\yng(4,1)$};
\node(53) at (2.5,-7){$\yng(1),\;\yng(5,1,1)$};
\node(61) at (-5,-8){$\yng(4,1,1,1,1),\;\emp$};
\node(62) at (-2.5,-8){$\yng(3,1,1,1),\;\yng(2)$};
\node(63) at (0,-8){$\yng(2,1,1),\;\yng(3,1)$};
\node(64) at (2.5,-8){$\yng(1,1),\;\yng(4,1,1)$};
\node(65) at (5,-8){$\emp,\;\yng(5,1,1,1)$};
\node(71) at (-2.5,-9){$\yng(3,1,1,1,1),\;\yng(1)$};
\node(72) at (0,-9){$\yng(2,1,1,1),\;\yng(2,1)$};
\node(73) at (2.5,-9){$\yng(1,1,1),\;\yng(3,1,1)$};
\node(81) at (-5,-10){$\yng(3,1,1,1,1,1),\emp$};
\node(82) at (0,-10){$\yng(2,1,1,1,1),\;\yng(1,1)$};
\node(83) at (2.5,-10){$\yng(1,1,1,1),\;\yng(2,1,1)$};
\node(84) at (5,-10){$\emp,\;\yng(4,1,1,1,1)$};
\node(91) at (2.5,-11){$\yng(1,1,1,1,1),\;\yng(1,1,1)$};
\node(101) at (-5,-12){$\yng(2,1,1,1,1,1,1),\;\emp$};
\node(102) at (5,-12){$\emp,\;\yng(3,1,1,1,1,1)$};
\node(111) at (-5,-14){$\yng(1,1,1,1,1,1,1,1),\;\emp$};
\node(112) at (5,-14){$\emp,\;\yng(2,1,1,1,1,1,1)$};
\node(121) at (5,-16){$\emp,\;\yng(1,1,1,1,1,1,1,1)$};
\path[->,semithick,RedOrange]
(11) edge node {}(01)
(21) edge node {}(11)
(22) edge node {}(12)
(41) edge node {}(21)
(31) edge node {}(01)
(31) edge node {}(12)
(42) edge node {}(11)
(42) edge node {}(31)
(43) edge node {}(31)
(43) edge node {}(22)
(44) edge node {}(22)
(51) edge node {}(21)
(51) edge node {}(42)
(52) edge node {}(42)
(52) edge node {}(43)
(53) edge node {}(43)
(53) edge node {}(44)
(61) edge node {}(41)
(62) edge node {}(41)
(62) edge node {}(51)
(63) edge node {}(51)
(63) edge node {}(52)
(64) edge node {}(52)
(64) edge node {}(53)
(65) edge node {}(53)
(71) edge node {}(61)
(71) edge node {}(62)
(72) edge node {}(62)
(72) edge node {}(63)
(73) edge node {}(63)
(73) edge node {}(64)
(81) edge node {}(71)
(82) edge node {}(71)
(82) edge node {}(72)
(83) edge node {}(72)
(83) edge node {}(73)
(84) edge node {}(64)
(84) edge node {}(65)
(91) edge node {}(82)
(91) edge node {}(83)
(101) edge node {}(81)
(101) edge node {}(82)
(102) edge node {}(73)
(102) edge node {}(84)
(111) edge node {}(101)
(111) edge node {}(91)
(112) edge node {}(83)
(112) edge node {}(102)
(121) edge node {}(112)
(121) edge node {}(91);
\path[->,semithick,OliveGreen]
(65) edge node {}(44)
(81) edge node {}(61);
\end{tikzpicture}}
\end{center}

\subsection{The principal block of $\OO_{1/3}(G(3,1,3))$}
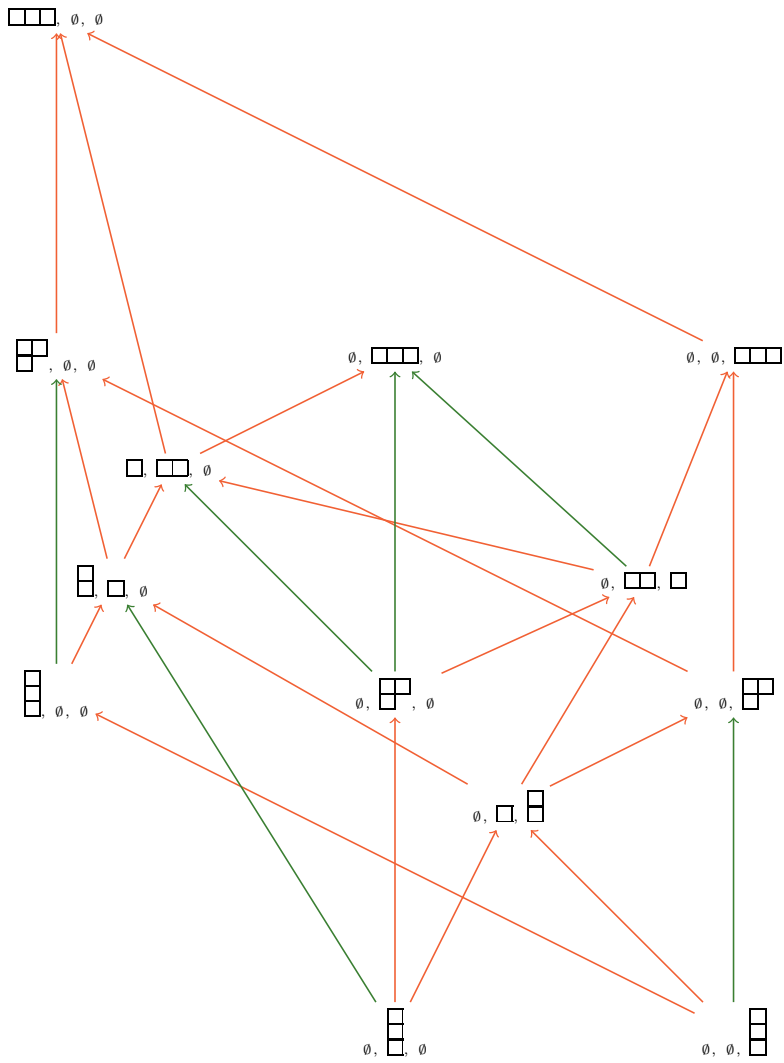
\label{G(3)}

\begin{center}
{\fontsize{5}{6}\selectfont
\begin{tikzpicture}[scale=1.5,every text node part/.style={align=center}]
\node(N11) at (-3,5){$\yng(3),\;\emp,\;\emp$};
\node(N21) at (-3,2){$\yng(2,1),\;\emp,\;\emp$};
\node(N22) at (0,2){$\emp,\;\yng(3),\;\emp$};
\node(N23) at (3,2){$\emp,\;\emp,\;\yng(3)$};
\node(N31) at (-2,1){$\yng(1),\;\yng(2),\;\emp$};
\node(N41) at (-2.5,0){$\yng(1,1),\;\yng(1),\;\emp$};
\node(N42) at (2.2,0){$\emp,\;\yng(2),\;\yng(1)$};
\node(N51) at (-3,-1){$\yng(1,1,1),\;\emp,\;\emp$};
\node(N52) at (0,-1){$\emp,\;\yng(2,1),\;\emp$};
\node(N53) at (3,-1){$\emp,\;\emp,\;\yng(2,1)$};
\node(N61) at (1,-2){$\emp,\;\yng(1),\;\yng(1,1)$};
\node(N71) at (0,-4){$\emp,\;\yng(1,1,1),\;\emp$};
\node(N72) at (3,-4){$\emp,\;\emp,\;\yng(1,1,1)$};
\path[->,semithick,RedOrange]
(N21) edge node {}(N11)
(N23) edge node {}(N11)
(N31) edge node {}(N11)
(N31) edge node {}(N22)
(N41) edge node {}(N21)
(N41) edge node {}(N31)
(N42) edge node {}(N31)
(N42) edge node {}(N23)
(N51) edge node {}(N41)
(N52) edge node {}(N42)
(N53) edge node {}(N21)
(N53) edge node {}(N23)
(N61) edge node {}(N41)
(N61) edge node {}(N42)
(N61) edge node {}(N53)
(N71) edge node {}(N52)
(N71) edge node {}(N61)
(N72) edge node {}(N51)
(N72) edge node {}(N61);
\path[->,semithick,OliveGreen]
(N42) edge node {}(N22)
(N51) edge node {}(N21)
(N52) edge node {}(N22)
(N52) edge node {}(N31)
(N71) edge node {}(N41)
(N72) edge node {}(N53);
\end{tikzpicture}}
\end{center}

\subsection{The (conjectural) BGG resolutions of $L(\Triv)$, $L(\emp,\emp,(3))$, and $L(\emp,(3),\emp)$ for $G(3,1,3)$ when $c=1/3$.}\label{G(3)ch}

\begin{center}
{\fontsize{5}{6}\selectfont
\begin{tikzpicture}[scale=1.5,every text node part/.style={align=center}]
\node(N0) at (0,0){$\yng(3),\;\emp,\;\emp$};
\node(N11) at (-2,-1.5){$\yng(2,1),\;\emp,\;\emp$};
\node(N12) at (0,-1.5){$\yng(1),\;\yng(2),\;\emp$};
\node(N13) at (2,-1.5){$\emp,\;\emp,\;\yng(3)$};
\node(N21) at (-2,-3){$\yng(1,1),\;\yng(1),\;\emp$};
\node(N22) at (2,-3){$\emp,\;\yng(2),\;\yng(1)$};
\node(N23) at (0,-3){$\emp,\;\emp,\;\yng(2,1)$};
\node(N3) at (0,-4.5){$\emp,\;\yng(1),\;\yng(1,1)$};
\path[->,semithick,RedOrange]
(N11) edge node {}(N0)
(N12) edge node {}(N0)
(N13) edge node {}(N0)
(N21) edge node {}(N11)
(N21) edge node {}(N12)
(N22) edge node {}(N12)
(N22) edge node {}(N13)
(N23) edge node {}(N11)
(N23) edge node {}(N13)
(N3) edge node {}(N21)
(N3) edge node {}(N22)
(N3) edge node {}(N23);
\end{tikzpicture}}
\end{center}

\begin{center}
{\fontsize{5}{6}\selectfont
\begin{tikzpicture}[scale=1.5,every text node part/.style={align=center}]
\node(x0) at (-1,0){$\emp,\;\yng(3),\;\emp$};
\node(x1) at (-2,-1){$\yng(1),\;\yng(2),\;\emp$};
\node(x2) at (-3,-2){$\yng(1,1),\;\yng(1),\;\emp$};
\node(x3) at (-4,-3){$\yng(1,1,1),\;\emp,\;\emp$};
\node(y01) at (3,0){$\emp,\;\emp,\;\yng(3)$};
\node(y11) at (2,-1){$\emp,\;\yng(2),\;\yng(1)$};
\node(y12) at (4,-1){$\emp,\;\emp,\;\yng(2,1)$};
\node(y21) at (1,-2){$\emp,\;\yng(2,1),\;\emp$};
\node(y22) at (3,-2){$\emp,\;\yng(1),\;\yng(1,1)$};
\node(y31) at (2,-3){$\emp,\;\yng(1,1,1),\;\emp$};
\path[->,semithick,RedOrange]
(x3) edge node {}(x2)
(x2) edge node {}(x1)
(x1) edge node {}(x0)
(y31) edge node {}(y21)
(y31) edge node {}(y22)
(y21) edge node {}(y11)
(y22) edge node {}(y11)
(y22) edge node {}(y12)
(y11) edge node {}(y01)
(y12) edge node {}(y01);
\end{tikzpicture}}
\end{center}

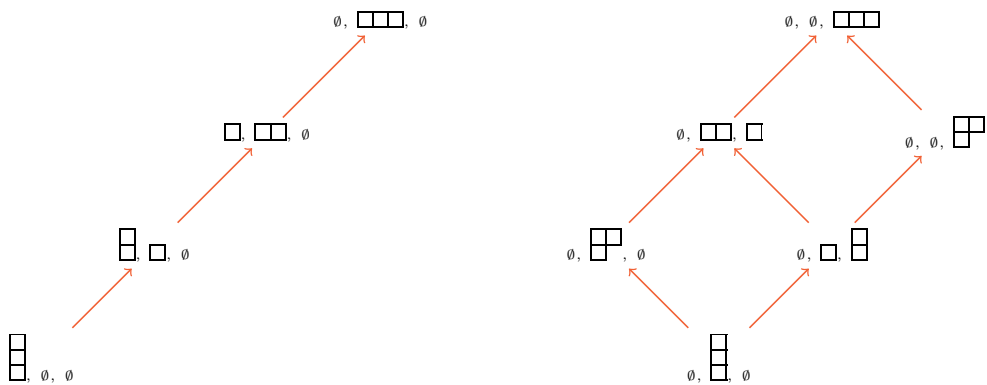
\subsection{The (conjectural) BGG resolution of $L(\Triv)$ for $G(3,1,6)$ when $c=1/6$.}\label{G(3,6)triv}
The graph $\bar{\Gamma}_{\Triv}$ encoding the resolution consists of eight cubes glued into a $2 \times 2 \times 2$ cube.
{\fontsize{5}{6}\selectfont
\begin{center}
\begin{tikzpicture}[scale=1.5,every text node part/.style={align=center}]
\node(01) at (0,0){$\yng(6),\;\emp,\;\emp$};
\node(11) at (-2.5,-1.5){$\yng(5,1),\;\emp,\;\emp$};
\node(12) at (0,-1.5){$\yng(3),\;\yng(3),\;\emp$};
\node(13) at (2.5,-1.5){$\yng(1),\;\emp,\;\yng(5)$};
\node(21) at (-5,-3){$\yng(4,1,1),\;\emp,\;\emp$};
\node(22) at (1,-3){$\yng(1,1),\;\emp,\;\yng(4)$};
\node(23) at (-3,-3){$\yng(3,1),\;\yng(2),\;\emp$};
\node(24) at (-1,-3){$\yng(2),\;\yng(3,1),\;\emp$};
\node(25) at (3,-3){$\yng(1),\;\yng(3),\;\yng(2)$};
\node(26) at (5,-3){$\emp,\;\emp,\;\yng(5,1)$};
\node(31) at (-1.7,-4.5){$\yng(1,1,1),\;\emp,\yng(3)$};
\node(32) at (0,-4.5){$\yng(1,1),\;\yng(2),\;\yng(2)$};
\node(33) at (3.3,-4.5){$\emp,\;\emp,\;\yng(4,1,1)$};
\node(34) at (-3.3,-4.5){$\yng(2,1),\;\yng(2,1),\;\emp$};
\node(35) at (5,-4.5){$\emp,\;\yng(3),\;\yng(2,1)$};
\node(36) at (1.7,-4.5){$\yng(1),\;\yng(3,1),\;\yng(1)$};
\node(37) at (-5,-4.5){$\yng(3,1,1),\;\yng(1),\;\emp$};
\node(41) at (1,-6){$\emp,\;\emp,\;\yng(3,1,1,1)$};
\node(42) at (3,-6){$\emp,\;\yng(2),\;\yng(2,1,1)$};
\node(43) at (-3,-6){$\yng(1,1,1),\;\yng(1),\;\yng(2)$};
\node(44) at (-5,-6){$\yng(2,1,1),\;\yng(1,1),\;\emp$};
\node(45) at (5,-6){$\emp,\;\yng(3,1),\;\yng(1,1)$};
\node(46) at (-1,-6){$\yng(1,1),\;\yng(2,1),\;\yng(1)$};
\node(51) at (0,-7.5){$\emp,\;\yng(1),\;\yng(2,1,1,1)$};
\node(52) at (2.5,-7.5){$\emp,\;\yng(2,1),\;\yng(1,1,1)$};
\node(53) at (-2.5,-7.5){$\yng(1,1,1),\;\yng(1,1),\;\yng(1)$};
\node(61) at (0,-9){$\emp,\;\yng(1,1),\;\yng(1,1,1,1)$};
\path[->,semithick,BurntOrange]
(11) edge node {}(01)
(12) edge node {}(01)
(13) edge node {}(01)
(22) edge node {}(11)
(23) edge node {}(11)
(23) edge node {}(12)
(25) edge node {}(12)
(22) edge node {}(13)
(25) edge node {}(13);
\path[->,semithick,Peach]
(24) edge node {}(12)
(36) edge node {}(25)
(34) edge node {}(24)
(36) edge node {}(24);
\path[->,semithick,RedOrange]
(21) edge node {}(11)
(26) edge node {}(13)
(31) edge node {}(21)
(37) edge node {}(21)
(37) edge node {}(23)
(32) edge node {}(25)
(35) edge node {}(25)
(33) edge node {}(26)
(35) edge node {}(26)
(42) edge node {}(32)
(46) edge node {}(32)
(42) edge node {}(35);
\path[->,semithick,Salmon]
(45) edge node {}(36)
(46) edge node {}(36)
(45) edge node {}(35)
(52) edge node {}(45);
\path[->,semithick,Orange]
(32) edge node {}(23)
(34) edge node {}(23)
(46) edge node {}(34)
(43) edge node {}(32)
(43) edge node {}(37)
(44) edge node {}(37)
(44) edge node {}(34)
(53) edge node {}(44);
\path[->,semithick,OrangeRed]
(31) edge node {}(22)
(32) edge node {}(22)
(33) edge node {}(22)
(41) edge node {}(31)
(41) edge node {}(33)
(42) edge node {}(33)
(43) edge node {}(31)
(51) edge node {}(41);
\path[->,semithick,CarnationPink]
(51) edge node {}(43)
(53) edge node {}(43)
(52) edge node {}(46)
(53) edge node {}(46)
(51) edge node {}(42)
(52) edge node {}(42)
(61) edge node {}(51)
(61) edge node {}(52)
(61) edge node {}(53);
\end{tikzpicture}
\end{center}
}

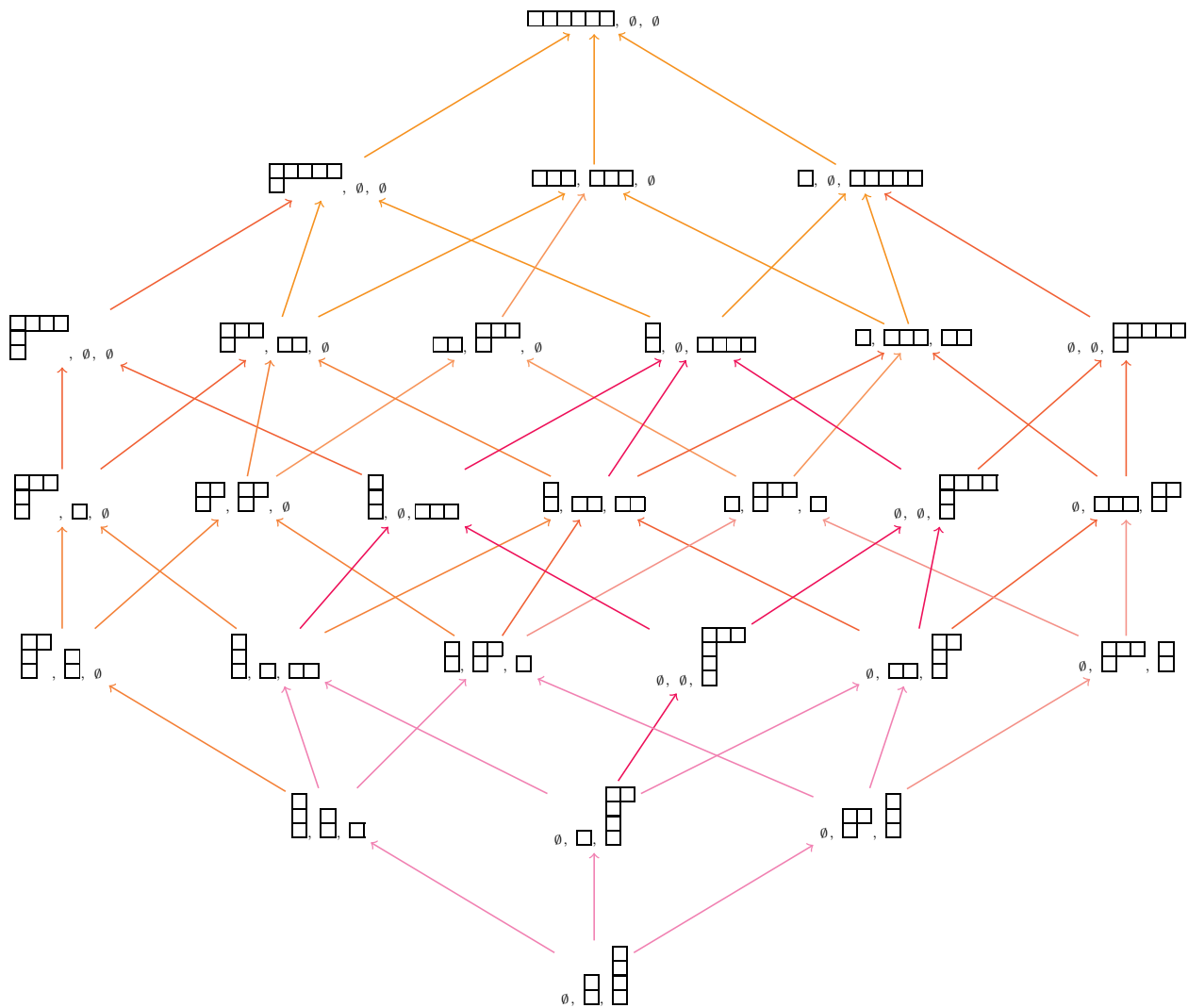
\subsection{The (conjectural) BGG resolutions of $L(\emp,(6),\emp)$ and $L(\emp,(6),\emp)$ for $G(3,1,6)$ at $c=1/6$. The principal block of $G(4,1,4)$ at $c=1/4$.}\label{G(3,6)ch,G(4)}

\begin{center}
{\fontsize{5}{6}\selectfont
\begin{tikzpicture}[scale=1.5,every text node part/.style={align=center}]
\node(11) at (2,0){$\emp,\;\yng(6),\;\emp$};
\node(21) at (.5,-1.5){$\yng(3),\;\yng(3),\;\emp$};
\node(22) at (2,-1.5){$\emp,\;\yng(5,1),\;\emp$};
\node(23) at (3.5,-1.5){$\emp,\;\yng(4),\;\yng(2)$};
\node(31) at (-1,-3){$\yng(3,1),\;\yng(2),\;\emp$};
\node(32) at (.5,-3){$\yng(2),\;\yng(3,1),\;\emp$};
\node(33) at (2,-3){$\yng(1),\;\yng(3),\;\yng(2)$};
\node(34) at (3.5,-3){$\emp,\yng(4,1),\;\yng(1)$};
\node(41) at (-2.5,-4.5){$\yng(3,1,1),\;\yng(1),\;\emp$};
\node(42) at (-1,-4.5){$\yng(2,1),\;\yng(2,1),\;\emp$};
\node(43) at (.5,-4.5){$\yng(1,1),\;\yng(2),\;\yng(2)$};
\node(44) at (2,-4.5){$\yng(1),\;\yng(3,1),\;\yng(1)$};
\node(51) at (-4,-6){$\yng(3,1,1,1),\;\emp,\;\emp$};
\node(52) at (-2.5,-6){$\yng(2,1,1),\;\yng(1,1),\;\emp$};
\node(53) at (-1,-6){$\yng(1,1,1),\;\yng(1),\;\yng(2)$};
\node(54) at (.5,-6){$\yng(1,1),\;\yng(2,1),\;\yng(1)$};
\node(61) at (-4,-7.5){$\yng(2,1,1,1,1),\;\emp,\;\emp$};
\node(62) at (-2.5,-7.5){$\yng(1,1,1,1),\;\emp,\yng(2)$};
\node(63) at (-1,-7.5){$\yng(1,1,1),\;\yng(1,1),\;\yng(1)$};
\node(71) at (-2.5,-9){$\yng(1,1,1,1,1),\;\emp,\;\yng(1)$};

\node(x01) at (0,-10){$\emp,\;\emp,\;\yng(6)$};
\node(x11) at (-1,-11){$\yng(1),\;\emp,\;\yng(5)$};
\node(x12) at (1,-11){$\emp,\;\yng(4),\;\yng(2)$};
\node(x21) at (-2,-12){$\yng(1,1),\;\emp,\yng(4)$};
\node(x22) at (0,-12){$\yng(1),\;\yng(3),\;\yng(2)$};
\node(x23) at (2,-12){$\emp,\;\yng(4,1),\;\yng(1)$};
\node(x31) at (-3,-13){$\yng(1,1,1),\;\emp,\;\yng(3)$};
\node(x32) at (-1,-13){$\yng(1,1),\;\yng(2),\;\yng(2)$};
\node(x33) at (1,-13){$\yng(1),\;\yng(3,1),\;\yng(1)$};
\node(x34) at (3,-13){$\emp,\;\yng(4,1,1),\;\emp$};
\node(x41) at (-2,-14){$\yng(1,1,1),\;\yng(1),\;\yng(2)$};
\node(x42) at (0,-14){$\yng(1,1),\;\yng(2,1),\;\yng(1)$};
\node(x43) at (2,-14){$\yng(1),\;\yng(3,1,1),\;\emp$};
\node(x51) at (-1,-15){$\yng(1,1,1),\;\yng(1,1),\;\yng(1)$};
\node(x52) at (1,-15){$\yng(1,1),\;\yng(2,1,1),\;\emp$};
\node(x61) at (0,-16){$\yng(1,1,1),\;\yng(1,1,1),\;\emp$};

\path[->,semithick,RedOrange]
(22) edge node {}(11)
(23) edge node {}(11)
(34) edge node {}(22)
(34) edge node {}(23);
\path[->,semithick,OrangeRed]
(21) edge node {}(11)
(32) edge node {}(21)
(32) edge node {}(22)
(33) edge node {}(21)
(33) edge node {}(23)
(44) edge node {}(32)
(44) edge node {}(33)
(44) edge node {}(34);
\path[->,semithick,RedOrange]
(31) edge node {}(21)
(42) edge node {}(31)
(42) edge node {}(32)
(43) edge node {}(31)
(43) edge node {}(33)
(54) edge node {}(43)
(54) edge node {}(42)
(54) edge node {}(44);
\path[->,semithick,BurntOrange]
(41) edge node {}(31)
(52) edge node {}(41)
(52) edge node {}(42)
(53) edge node {}(41)
(53) edge node {}(43)
(63) edge node {}(52)
(63) edge node {}(53)
(63) edge node {}(54);
\path[->,semithick,RedOrange]
(51) edge node {}(41)
(61) edge node {}(52)
(61) edge node {}(51)
(62) edge node {}(51)
(62) edge node {}(53)
(71) edge node {}(61)
(71) edge node {}(62)
(71) edge node {}(63);
\path[->,semithick,RedOrange]
(x61) edge node {}(x51)
(x61) edge node {}(x52)
(x51) edge node {}(x41)
(x51) edge node {}(x42)
(x52) edge node {}(x42)
(x52) edge node {}(x43)
(x41) edge node {}(x31)
(x41) edge node {}(x32)
(x42) edge node {}(x32)
(x42) edge node {}(x33)
(x43) edge node {}(x33)
(x43) edge node {}(x34)
(x31) edge node {}(x21)
(x32) edge node {}(x21)
(x32) edge node {}(x22)
(x33) edge node {}(x22)
(x33) edge node {}(x23)
(x34) edge node {}(x23)
(x21) edge node {}(x11)
(x22) edge node {}(x11)
(x22) edge node {}(x12)
(x23) edge node {}(x12)
(x11) edge node {}(x01)
(x12) edge node {}(x01);
\end{tikzpicture}}
\end{center}

{\fontsize{5}{6}\selectfont
\begin{center}
\begin{tikzpicture}[scale=1.5,every text node part/.style={align=center}]
\node(0) at (-1,3){$\yng(4),\;\emp,\;\emp,\;\emp$};
\node(11) at (-7,-1){$\yng(3,1),\;\emp,\;\emp,\;\emp$};
\node(12) at (-2,-1){$\emp,\;\yng(4),\;\emp,\;\emp$};
\node(13) at (1,-1){$\emp,\;\emp,\;\yng(4),\;\emp$};
\node(14) at (4,-1){$\emp,\;\emp,\;\emp,\;\yng(4)$};
\node(21) at (-5,-3){$\yng(2),\;\yng(2),\;\emp,\;\emp$};
\node(22) at (1,-3){$\yng(1),\;\emp,\;\yng(3),\;\emp$};
\node(31) at (-4.5,-4){$\yng(2,1),\;\yng(1),\;\emp,\;\emp$};
\node(32) at (-1,-4){$\emp,\;\yng(3),\;\yng(1),\;\emp$};
\node(33) at (3,-4){$\emp,\;\emp,\;\yng(3),\;\yng(1)$};
\node(41) at (-7,-5){$\yng(2,1,1),\;\emp,\;\emp,\;\emp$};
\node(42) at (-5.2,-5) {$\yng(1,1),\;\emp,\;\yng(2),\;\emp$};
\node(43) at (-3.4,-5) {$\emp,\;\yng(3,1),\;\emp,\;\emp$};
\node(44) at (-1.6,-5) {$\yng(1),\;\yng(2),\;\yng(1),\;\emp$};
\node(45) at (.2,-5) {$\emp,\;\emp,\;\yng(3,1),\;\emp$};
\node(46) at (2,-5) {$\emp,\;\yng(2),\;\emp,\;\yng(2)$};
\node(47) at (4,-5) {$\emp,\;\emp,\;\emp,\;\yng(3,1)$};
\node(51) at (-4,-6){$\yng(1,1),\;\yng(1),\;\yng(1),\;\emp$};
\node(52) at (-1.5,-6){$\yng(1),\;\yng(2,1),\;\emp,\;\emp$};
\node(53) at (0,-6){$\emp,\;\yng(2),\;\yng(1),\;\yng(1)$};
\node(61) at (-5.5,-7){$\yng(1,1,1),\;\emp,\;\yng(1),\;\emp$};
\node(62) at (-1,-7){$\emp,\;\yng(2),\;\yng(1,1),\;\emp$};
\node(63) at (-3,-7){$\yng(1,1),\;\yng(1,1),\;\emp,\;\emp$};
\node(64) at (.5,-7){$\emp,\;\emp,\;\yng(2),\;\yng(1,1)$};
\node(65) at (2,-7){$\emp,\;\yng(2,1),\;\emp,\;\yng(1)$};
\node(66) at (3.5,-7){$\emp,\;\yng(1),\;\emp,\;\yng(2,1)$};
\node(71) at (-.5,-8){$\emp,\;\yng(1),\;\yng(1),\;\yng(1,1)$};
\node(81) at (-7,-9){$\yng(1,1,1,1),\;\emp,\;\emp,\;\emp$};
\node(82) at (-4,-9){$\emp,\;\yng(2,1,1),\;\emp,\;\emp$};
\node(83) at (-1,-9) {$\emp,\;\yng(1,1),\;\emp,\;\yng(1,1)$};
\node(84) at (2,-9){$\emp,\;\emp,\;\yng(2,1,1),\;\emp$};
\node(85) at (4,-9){$\emp,\;\emp,\;\emp,\;\yng(2,1,1)$};
\node(91) at (-2.5,-10){$\emp,\;\yng(1),\;\yng(1,1,1),\;\emp$};
\node(92) at (.5,-10){$\emp,\;\emp,\;\yng(1),\;\yng(1,1,1)$};
\node(10a) at (-4,-13){$\emp,\;\yng(1,1,1,1),\;\emp,\;\emp$};
\node(10b) at (-1,-13){$\emp,\;\emp,\;\yng(1,1,1,1),\;\emp$};
\node(10c) at (2,-13){$\emp,\;\emp,\;\emp,\;\yng(1,1,1,1)$};
\path[->,semithick,RedOrange]
(11) edge node {}(0)
(14) edge node {}(0)
(21) edge node {}(0)
(21) edge node {}(12)
(22) edge node {}(0)
(22) edge node {}(13)
(31) edge node {}(11)
(31) edge node {}(21)
(32) edge node {}(12)
(32) edge node {}(13)
(33) edge node {}(22)
(33) edge node {}(14)
(41) edge node {}(31)
(42) edge node {}(11)
(42) edge node {}(22)
(43) edge node {}(32)
(44) edge node {}(21)
(44) edge node {}(22)
(44) edge node {}(32)
(45) edge node {}(33)
(46) edge node {}(14)
(46) edge node {}(21)
(47) edge node {}(11)
(47) edge node {}(14)
(51) edge node {}(31)
(51) edge node {}(42)
(51) edge node {}(44)
(52) edge node {}(43)
(52) edge node {}(44)
(53) edge node {}(44)
(53) edge node {}(46)
(53) edge node {}(33)
(61) edge node {}(41)
(61) edge node {}(51)
(62) edge node {}(53)
(62) edge node {}(45)
(63) edge node {}(51)
(63) edge node {}(52)
(64) edge node {}(33)
(64) edge node {}(42)
(64) edge node {}(47)
(65) edge node {}(53)
(65) edge node {}(52)
(66) edge node {}(46)
(66) edge node {}(47)
(66) edge node {}(31)
(71) edge node {}(51)
(71) edge node {}(53)
(71) edge node {}(64)
(71) edge node {}(66)
(81) edge node {}(61)
(81) edge node {}(63)
(82) edge node {}(62)
(82) edge node {}(65)
(83) edge node {}(71)
(83) edge node {}(63)
(83) edge node {}(65)
(84) edge node {}(45)
(84) edge node {}(64)
(85) edge node {}(41)
(85) edge node {}(66)
(91) edge node {}(84)
(91) edge node {}(62)
(91) edge node {}(71)
(92) edge node {}(61)
(92) edge node {}(71)
(92) edge node {}(85)
(10a) edge node {}(82)
(10a) edge node {}(83)
(10a) edge node {}(91)
(10b) edge node {}(91)
(10b) edge node {}(92)
(10c) edge node {}(81)
(10c) edge node {}(83)
(10c) edge node {}(92);
\path[->,semithick,OliveGreen]
(33) edge node {}(13)
(41) edge node {}(11)
(43) edge node {}(12)
(45) edge node {}(22)
(45) edge node {}(13)
(46) edge node {}(12)
(52) edge node {}(21)
(53) edge node {}(32)
(61) edge node {}(42)
(62) edge node {}(44)
(62) edge node {}(32)
(63) edge node {}(31)
(65) edge node {}(43)
(65) edge node {}(46)
(81) edge node {}(41)
(82) edge node {}(52)
(82) edge node {}(43)
(83) edge node {}(66)
(84) edge node {}(42)
(85) edge node {}(47)
(91) edge node {}(51)
(92) edge node {}(64)
(10a) edge node {}(63)
(10b) edge node {}(61)
(10b) edge node {}(84)
(10c) edge node {}(85);
\end{tikzpicture}
\end{center}}

\subsection{The (conjectural) BGG resolution of $L(\Triv)$ for $G(4,1,4)$ when $c=1/4$.}\label{G(4)triv}
The graph $\bar{\Gamma}_{\Triv}$ encoding the resolution is a hypercube. 
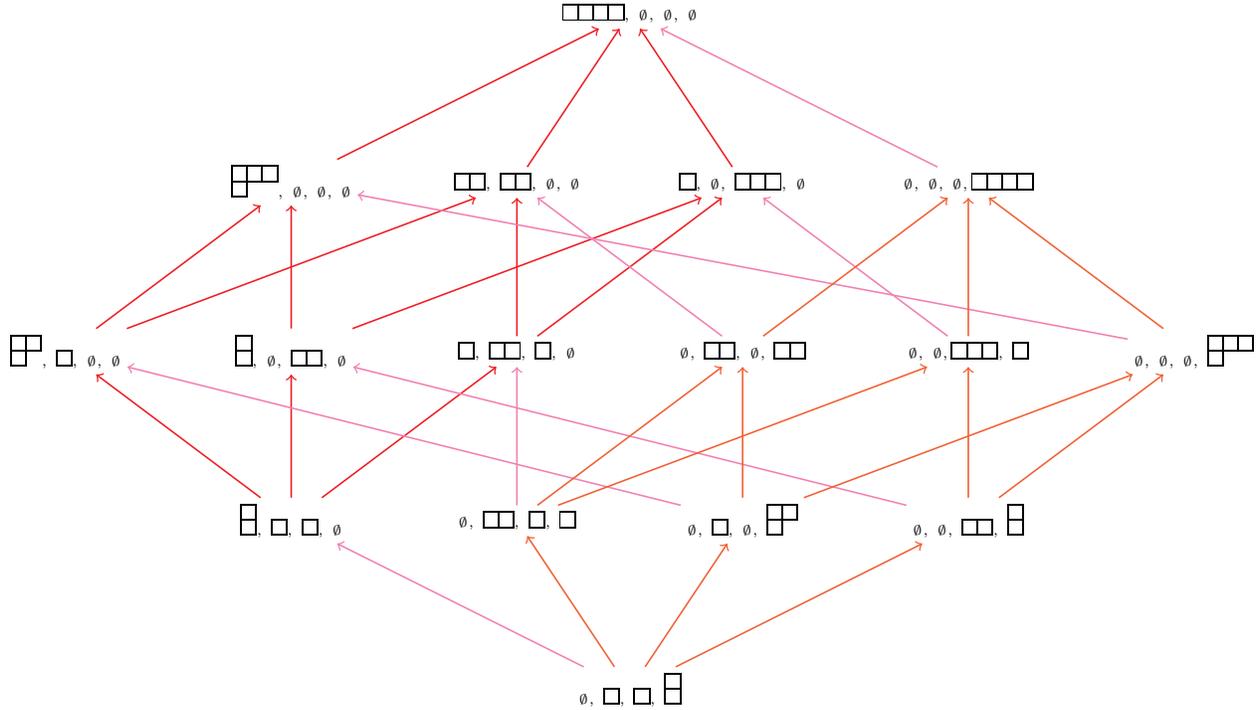
\begin{center}
{\fontsize{5}{6}\selectfont
\begin{tikzpicture}[scale=1.5,every text node part/.style={align=center}]
\node(N0) at (0,0){$\yng(4),\;\emp,\;\emp,\;\emp$};
\node(N11) at (-3,-1.5){$\yng(3,1),\;\emp,\;\emp,\;\emp$};
\node(N12) at (-1,-1.5){$\yng(2),\;\yng(2),\;\emp,\;\emp$};
\node(N13) at (1,-1.5){$\yng(1),\;\emp,\;\yng(3),\;\emp$};
\node(N14) at (3,-1.5){$\emp,\;\emp,\;\emp,\yng(4)$};
\node(N21) at (-5,-3){$\yng(2,1),\;\yng(1),\;\emp,\;\emp$};
\node(N22) at (-3,-3){$\yng(1,1),\;\emp,\;\yng(2),\;\emp$};
\node(N23) at (-1,-3){$\yng(1),\;\yng(2),\;\yng(1),\;\emp$};
\node(N24) at (1,-3){$\emp,\;\yng(2),\;\emp,\;\yng(2)$};
\node(N25) at (3,-3){$\emp,\;\emp,\yng(3),\;\yng(1)$};
\node(N26) at (5,-3){$\emp,\;\emp,\;\emp,\;\yng(3,1)$};
\node(N31) at (-3,-4.5){$\yng(1,1),\;\yng(1),\;\yng(1),\;\emp$};
\node(N32) at (-1,-4.5){$\emp,\;\yng(2),\;\yng(1),\;\yng(1)$};
\node(N33) at (1,-4.5){$\emp,\;\yng(1),\;\emp,\;\yng(2,1)$};
\node(N34) at (3,-4.5){$\emp,\;\emp,\;\yng(2),\;\yng(1,1)$};
\node(N4) at (0,-6){$\emp,\;\yng(1),\;\yng(1),\;\yng(1,1)$};
\path[->,semithick,Red]
(N11) edge node {}(N0)
(N12) edge node {}(N0)
(N13) edge node {}(N0)
(N21) edge node {}(N11)
(N21) edge node {}(N12)
(N22) edge node {}(N11)
(N22) edge node {}(N13)
(N23) edge node {}(N12)
(N23) edge node {}(N13)
(N31) edge node {}(N21)
(N31) edge node {}(N22)
(N31) edge node {}(N23);
\path[->,semithick,CarnationPink]
(N14) edge node {}(N0)
(N25) edge node {}(N13)
(N24) edge node {}(N12)
(N26) edge node {}(N11)
(N32) edge node {}(N23)
(N34) edge node {}(N22)
(N33) edge node {}(N21)
(N4) edge node {}(N31);
\path[->,semithick,RedOrange]
(N24) edge node {}(N14)
(N25) edge node {}(N14)
(N26) edge node {}(N14)
(N32) edge node {}(N24)
(N32) edge node {}(N25)
(N33) edge node {}(N24)
(N33) edge node {}(N26)
(N34) edge node {}(N25)
(N34) edge node {}(N26)
(N4) edge node {}(N32)
(N4) edge node {}(N33)
(N4) edge node {}(N34);
\end{tikzpicture}}
\end{center}

\subsection{An informative (non)-example.}

Consider $\OO_c(B_2)$ when $c_0=1/2$, $d_0=1$. There are five simples labeled by the five bipartitions of $2$ and they all belong to the same block. By Lemma \ref{diag} every $\Delta(\lambda)$ is $\ttt$-diagonalizable. $L_c(\Triv)$ is finite-dimensional and has a basis of Jack polynomials by \cite{Gri}.

On the left, the graph $\Gamma$; on the right, the $\Ext^1$ quiver of the block, $Q_{\Ext^1}$:
\begin{center}
{\fontsize{5}{6}\selectfont
\begin{tikzpicture}[scale=1,every text node part/.style={align=center}]
\node(N1) at (-4,0){$\yng(2),\;\emp$};
\node(N2) at (-2,-1){$\yng(1,1),\;\emp$};
\node(N3) at (-4,-2){$\yng(1),\;\yng(1)$};
\node(N4) at (-2,-3){$\emp,\;\yng(2)$};
\node(N5) at (-4,-4){$\emp,\;\yng(1,1)$};

\node(Q1) at (2,0){$\yng(2),\;\emp$};
\node(Q2) at (4,-1){$\yng(1,1),\;\emp$};
\node(Q3) at (2,-2){$\yng(1),\;\yng(1)$};
\node(Q4) at (4,-3){$\emp,\;\yng(2)$};
\node(Q5) at (2,-4){$\emp,\;\yng(1,1)$};

\path[->,semithick,RedOrange]
(N5) edge node {}(N4)
(N4) edge node {}(N3)
(N3) edge node {}(N2)
(N2) edge node {}(N1);
\path[->,semithick,OliveGreen]
(N5) edge node {}(N3)
(N3) edge node {}(N1);
\path[<->,semithick,RedOrange]
(Q5) edge node {}(Q4)
(Q4) edge node {}(Q3)
(Q3) edge node {}(Q2)
(Q2) edge node {}(Q1)
(Q4) edge node {}(Q1);
\end{tikzpicture}}
\end{center} 
We see that $Q_{\Ext^1}$ does not coincide with the double of $Q_{\mathrm{prim}}$, which is the subquiver of $\Gamma$ consisting of the orange arrows. Therefore $\OO_c(B_2)$ cannot have tight multiplicities by Corollary \ref{quiver criterion}. However, the Serre subcategory of $\OO_c(B_2)$ spanned by $\{L_c(\tau)\;|\;\tau\leq((1^2),\;\emp)\}$ does have tight multiplicities since the subquivers of $Q_{\Ext^1}$ and $Q_{\mathrm{prim}}$ coincide there, by Corollary \ref{quiver criterion} again. It follows that the only simple in $\OO_c(B_2)$ which does not have a BGG resolution is $L_c(\Triv)$.

\end{document}